\documentclass[reqno]{amsart}

\usepackage{amsmath,amsfonts,amsthm,amssymb,graphics, paralist, pstricks}

\newcommand{\g}{\geqslant}
\newcommand{\ar}{\rangle}
\newcommand{\al}{\langle}

\newcommand{\RR}{\mathbb{R}}

\newcommand{\CC}{\mathbb{C}}

\newcommand{\NN}{\mathbb{N}}
\newcommand{\p}{\partial}
\newcommand{\q}{\varphi}
\newcommand{\les}{\leqslant}
\newcommand{\lesa}{\lesssim}

\newcommand{\supp}{\text{supp }\,}

\DeclareSymbolFont{bbold}{U}{bbold}{m}{n}
\DeclareSymbolFontAlphabet{\mathbbold}{bbold}

\newcommand{\ind}{\mathbbold{1}}

\theoremstyle{plain}
\newtheorem{theorem}{Theorem}
\newtheorem{proposition}[theorem]{Proposition}
\newtheorem{lemma}[theorem]{Lemma}
\newtheorem{corollary}[theorem]{Corollary}

\theoremstyle{remark}
\newtheorem{remark}{Remark}

\setdefaultenum{(i)}{(a)}{1}{A}
\title[Bilinear Estimates and applications to GWP for the DKG equation]{Bilinear Estimates and Applications to Global Well-Posedness for the Dirac-Klein-Gordon equation on $\RR^{1+1}$}
\author[Timothy Candy]{Timothy Candy \vspace{0.4cm}\\
\textit{D\lowercase{epartment} \lowercase{of} M\lowercase{athematics}, U\lowercase{niversity} \lowercase{of} E\lowercase{dinburgh}\\
E\lowercase{dinburgh} EH9 3JE, U\lowercase{nited} K\lowercase{ingdom} \\
E\lowercase{mail}: T.L.C\lowercase{andy}@\lowercase{sms.ed.ac.uk}}}

\thanks{The author wishes to thank his supervisor Nikolaos Bournaveas for many helpful conversations related to this work. The results in this article will form part of the authors PhD thesis.}

\date{\today}
\begin{document}
\maketitle
\begin{abstract}
  We prove new bilinear estimates for the  $X^{s, b}_{\pm}(\RR^2)$ spaces which are optimal up to endpoints. These estimates are often used in the theory of nonlinear Dirac equations on $\RR^{1+1}$. The proof of the bilinear estimates follows from a dyadic decomposition in the spirit of  Tao \cite{Tao2001} and D'Ancona, Foschi, and Selberg \cite{D'Ancona2010}. As an application, by using the $I$-method of Colliander, Keel, Staffilani, Takaoka, and Tao,  we extend the work of Tesfahun \cite{Tesfahun2009} on global existence below the charge class for the Dirac-Klein-Gordon equation on $\RR^{1+1}$.

\end{abstract}
\setcounter{tocdepth}{1}

\tableofcontents

\section{Introduction}
 We consider the problem of proving bilinear estimates in the Bourgain-Klainerman-Machedon type spaces $X^{s, b}_{\pm}$ on $\RR^2$, where we define the spaces $X^{s, b}_\pm$ via the norm
        $$ \| \psi \|_{X^{s, b}_\pm} = \big\| \al \tau \pm \xi\ar^b \al \xi \ar^s \widetilde{\psi}(\tau, \xi) \|_{L^2_{\tau, \xi}(\RR^2)}$$
 with $\al \cdot \ar= \sqrt{ 1+ |\cdot|^2}$. These spaces have been used in the low regularity theory of various nonlinear Dirac equations in one space dimension, \cite{Machihara2005, Selberg2010b}, as well as the Dirac-Klein-Gordon (DKG) system \cite{Pecher2006, Selberg2008}. Though recently, product Sobolev spaces based on the null coordinates $x\pm t$ have also proved useful \cite{Candy2010, Machihara2010}. In applications of the $X^{s, b}_\pm$ spaces to  low regularity well-posedness, we often require product estimates of the form
            \begin{equation}\label{standard form} \| u v\|_{X^{ -s_1, -b_1}_{\pm_1}} \lesa \| u \|_{X^{s_2, b_2}_{\pm_2}} \| v \|_{X^{s_3, b_3}_{\pm_3}}\end{equation}
 where $s_j, b_j \in \RR$ and $\pm_j$ are independent choices of $\pm$. A number of estimates of this form, for specific values of $s_j$ and $b_j$, have appeared previously in the literature \cite{Machihara2005, Selberg2008, Selberg2010b}. The case where $\pm_1 = \pm_2 = \pm_3$ is not particularly interesting, as a simple change of variables reduces (\ref{standard form}) to two applications of the 1-dimensional Sobolev product estimate
        $$ \| f g\|_{H^{-s_1}(\RR)} \lesa \| f\|_{H^{s_2}(\RR)} \| g\|_{H^{s_3}(\RR)}.$$
 Thus leading to the conditions\footnote{For the sake of exposition, we are ignoring the endpoint cases. The sharp result allows one of the inequalities in (\ref{product sobolev cond 1}) to replaced with an equality, a similar comment applies to the condition (\ref{product sobolev cond 2}).}
        \begin{equation}
          \label{product sobolev cond 1} b_j + b_k>0, \qquad b_1 + b_2 + b_3>\frac{1}{2}
        \end{equation}
 and
        \begin{equation}
          \label{product sobolev cond 2} s_j + s_k>0, \qquad s_1 + s_2 + s_3>\frac{1}{2}
        \end{equation}
 where $j \neq k$. On the other hand, if we have $\pm_1=\pm_2=\pm$ and $\pm_3 = \mp$, then we can make significant improvements over  (\ref{product sobolev cond 2}).  This observation allows one to exploit the null structure that is often found in nonlinear hyperbolic systems in one dimension, see for instance \cite{Selberg2010b}.\\

To state our first result we use the following conventions. For a set of real numbers $\{a_1, a_2, a_3\}$, we let $a_{max}=\max_{i} a_i$, $a_{min}=\min_i a_i$, and use $a_{med}$ to denote the median. If $a \in \RR$ then we define
            $$ a_+ = \begin{cases} a \qquad &a>0 \\
                                    0  &a\les 0. \end{cases} $$
We state our product estimate in the dual form.

\begin{theorem}\label{thm main X^sb esimate}
  Let $s_j$, $b_j \in \RR$, $j=1, 2, 3$ satisfy
   \begin{equation} b_1+ b_2 + b_3 >\frac{1}{2},\qquad b_j + b_k >0,  \qquad (j\neq k) \label{thm main X^sb esitimate - cond on b}\end{equation}
    and for $k \in \{ 1, 2\}$
      \begin{equation} \label{thm main X^sb esitimate - cond on s} \begin{split} s_1 + s_2 &\g 0, \\
       s_k + s_3  &> -b_{min},\\
       s_k +s_3&> \frac{1}{2} - b_1 - b_2 - b_3, \\
       s_1 + s_2 + s_3 &>  \frac{1}{2} - b_3,\\
       s_1 + s_2 + s_3 &> \Big(\frac{1}{2} - b_{max}\Big)_+ + \Big(\frac{1}{2} - b_{med}\Big)_+ - b_{min}. \end{split}\end{equation}
  Then
             \begin{equation}\label{thm main X^sb est - main trilinear est} \Big|\int_{\RR^2} \Pi_{j=1}^3 \psi_j(t, x) dx dt\Big| \lesa \| \psi_1\|_{X^{s_1, b_1}_\pm} \| \psi_2 \|_{X^{s_2, b_2}_\pm} \| \psi_3 \|_{X^{s_3, b_3}_\mp}. \end{equation}
  Moreover the conditions (\ref{thm main X^sb esitimate - cond on b}) and (\ref{thm main X^sb esitimate - cond on s}) are sharp up to equality.
\end{theorem}

\begin{remark}
 There are cases where we can allow  equality in (\ref{thm main X^sb esitimate - cond on b}) or (\ref{thm main X^sb esitimate - cond on s}), for instance the case
    $$s_1=s_2=s_3=0, \qquad  b_1=0,\qquad  b_2=b_3=\frac{1}{2}+\epsilon$$
 holds  \cite[Corollary 1]{Selberg2008}. We have not attempted to list or prove the endpoint cases here, as this would significantly complicate  the statement of Theorem \ref{thm main X^sb esimate}. Additionally, Theorem \ref{thm main X^sb esimate} is sufficient for our intended application to global well-posedness for the Dirac-Klein-Gordon equation.
\end{remark}

Define the Wave-Sobolev spaces $H^{s, b}$ by using the norm
        $$ \| \psi\|_{H^{s, b}} = \big\| \al |\tau| - |\xi| \ar^b \al \xi \ar^s \widetilde{\psi}(\tau, \xi) \|_{L^2_{\tau, \xi}(\RR^2)}.$$
Then as a simple corollary to Theorem \ref{thm main X^sb esimate} we can replace one of the $X^{s, b}_{\pm}$ norms on the righthand side of (\ref{thm main X^sb est - main trilinear est}) with a $H^{s, b}$ norm.

\begin{corollary}\label{cor wave-sobolev and X^sb estimate}
Let $r, s_1, s_2,  b_j \in \RR$, $j=1, 2, 3$ satisfy
    $$ b_1+ b_2 + b_3>\frac{1}{2}, \qquad  b_j + b_k>0, \qquad (j \neq k)$$
and for $k \in \{ 1, 2\}$
      \begin{align*}  s_k +r &\g 0, \\
       s_k + r &>-b_{min} \\
       s_1 + s_2 &> -b_{min},\\
       s_1 + s_2 &> \frac{1}{2} - b_1 - b_2 - b_3, \\
       s_1 + s_2 + r &>  \frac{1}{2} - b_k ,\\
       s_1 + s_2 + r&> \Big(\frac{1}{2} - b_{max}\Big)_+ + \Big(\frac{1}{2} - b_{med}\Big)_+ - b_{min}. \end{align*}
  Then
             $$ \Big|\int_{\RR^2} \Pi_{j=1}^3 \psi_j(t, x) dx dt\Big| \lesa  \| \psi_1 \|_{X^{s_1, b_1}_+} \| \psi_2 \|_{X^{s_2, b_2}_-}\| \psi_3\|_{H^{r, b_3}}. $$
\begin{proof}
  We decompose $\psi_3$ into the regions $\{ (\tau, \xi) \in \RR^{1+1} \,| \, \pm \tau \xi \g 0 \}$  and observe that on the first region $\al |\tau| - |\xi| \ar  = \al \tau - \xi \ar$ while in the second region $ \al |\tau| - |\xi| \ar  = \al \tau + \xi \ar$. The corollary now follows from two applications of Theorem \ref{thm main X^sb esimate}.
\end{proof}
\end{corollary}
\begin{remark}
This result should be compared to the similar estimates  contained in \cite{Selberg2008} and \cite{Tesfahun2009}. Also we note that the decomposition used in the proof of Corollary \ref{cor wave-sobolev and X^sb estimate} can  be used to give bilinear estimates in the Wave-Sobolev spaces $H^{r, b}$, thus giving an alternative (though closely related) proof of Theorem 7.1 in \cite{D'Ancona2010b} (up to endpoints).
\end{remark}

The second main result contained in this article concerns the global existence problem for the DKG equation on $\RR^{1+1}$. The DKG equation can be written as
        \begin{equation}\label{DKG general form}
            \begin{split}
                \big( \gamma^0 \p_t + \gamma^1 \p_x\big) \psi &= -i M \psi + i \phi \psi \\
                \big( - \Box + m^2\big)\phi &= \al \gamma^0 \psi, \psi \ar_{\CC^2}
            \end{split}
        \end{equation}
with initial data
        \begin{equation}\label{DKG general initial data}
            \psi(0) = \psi_0 \in H^s, \qquad \phi(0) = \phi_0 \in H^r, \qquad \p_t \phi(0) = \phi_1 \in H^{r-1}
        \end{equation}
for some values of $s, r \in \RR$. The d'Alembertian is defined by $\Box= - \p_t^2 + \p_x^2$ and  we take the standard representation of the Dirac matrices
        $$ \gamma^0 = \begin{pmatrix} 0 & 1 \\ 1 & 0 \end{pmatrix}, \qquad \gamma^{1} = \begin{pmatrix} 0 & -1 \\ 1 & 0 \end{pmatrix}.$$
The Dirac spinor $\psi \in \CC^2$, and the real-valued scalar field $\phi \in \RR$, are functions of $(t, x) \in \RR^{1+1}$. The notation $\al \cdot, \cdot\ar_{\CC^2}$ refers to the standard inner product on $\CC^2$, and  $m, M \in \RR$ are constants.

There are two main features of the DKG equation (\ref{DKG general form}) which we wish to highlight here. The first feature concerns the conservation of charge which can be stated as follows: if $(\psi, \phi)$ is a smooth solution to (\ref{DKG general form}) with sufficient decay at infinity, then for all times $t \in \RR$ we have
        \begin{equation}\label{conservation of charge} \| \psi(t) \|_{L^2} = \| \psi(0) \|_{L^2}. \end{equation}
The conservation of charge is crucial in controlling the global behaviour of the solution $(\psi, \phi)$. The second feature we would like to note is that the nonlinearity in the DKG equation has null structure. Roughly speaking, this refers to the fact that the nonlinear terms in (\ref{DKG general form}) behave significantly better than generic products. The null structure is a crucial component in the low regularity existence theory for the DKG equation and has been used by a number of authors \cite{Bournaveas2006, Fang2004a, Machihara2007b, Pecher2006, Selberg2008}. The observation that null structure can be used to improve local existence results for nonlinear wave equations is due to Klainerman and Machedon in \cite{Klainerman1993}. \\

The question of local well-posedness (LWP) for the DKG equation was first considered by Chadam \cite{Chadam1973}. Subsequently, much progress has been made by numerous authors \cite{Bournaveas2006, Fang2004a, Machihara2007b, Pecher2006, Selberg2008}. The best result to date is due to Machihara, Nakanishi, and Tsugawa \cite{Machihara2010} where it was shown that (\ref{DKG general form}) with initial data (\ref{DKG general initial data}) is locally well-posed provided
        $$ s>-\frac{1}{2}, \qquad |s| \les r \les s+1.$$
Moreover, this region is essentially sharp, except possibly at the endpoint $s= - \frac{1}{2}$. More precisely, outside this region the solution map is either ill-posed, or fails to be twice differentiable, see \cite{Machihara2010} for  a more precise statement.\\

In the current article we are interested in the minimum regularity required on the initial data (\ref{DKG general initial data}) to ensure that the corresponding local in time solution $(\psi, \phi)$ to (\ref{DKG general form}) can be extended globally in time. Global well-posedness (GWP) in the high regularity case $s=r=1$ was first proven by Chadam \cite{Chadam1973}, this was then progressively lowered to $s\g0$ by a number of authors \cite{Bournaveas2000, Bournaveas2006a, Chadam1973, Fang2004a, Pecher2006} by exploiting the conservation of charge (\ref{conservation of charge}) together with the local well-posedness theory. The first result below the charge class was due to Selberg \cite{Selberg2007} where it was shown that the DKG equation is GWP in the region\footnote{Note that this also gives GWP in the region $s>0$, $|s|\les r \les s+1$ by persistence of regularity, see for instance \cite{Selberg2008}.}
    $$- \frac{1}{8}<s<0, \qquad -s+\sqrt{s^2 - s} < r\les s+1.$$
Note that when $s<0$, the conservation of charge cannot be used directly since $\psi \not \in L^2$, thus the problem of global existence is significantly more difficult. Instead Selberg made use of the Fourier truncation method of Bourgain \cite{Bourgain1998}, which allows one to take initial data just below a conserved quantity.  There is a difficulty in directly applying this method to the DKG equation however, as there is no conservation law for the scalar $\phi$. Instead, one needs to exploit the fact the nonlinearity for $\phi$ depends only on the spinor $\psi$. Thus, as we have control over $\psi$ via the conservation of charge, we should be able to estimate the growth of $\phi$. This strategy was implemented by Selberg via an induction argument involving the cascade of free waves.

Currently, the best result for GWP for the DKG equation is due to Tesfahun \cite{Tesfahun2009} where the GWP region of Selberg was extended to
          $$- \frac{1}{8}<s<0, \qquad s+\sqrt{s^2 - s} < r\les s+1.$$
The improvement comes from applying the $I$-method of Colliander, Keel, Staffilani, Takaoka, and Tao, see for instance \cite{Colliander2002} for an introduction to the $I$-method. In the current article, we prove the following.

\begin{theorem}\label{thm DKG in intro}
The DKG equation (\ref{DKG general form}) is globally well-posed for initial data $ \psi_0 \in H^s$, $(\phi_0, \phi_1) \in H^r\times H^{r-1}$ provided
 $$  -\frac{1}{6}<s<0,  \qquad s-\frac{1}{4} + \sqrt{\Big(s-\frac{1}{4}\Big)^2 - s}<r \les s+1.$$
\end{theorem}

 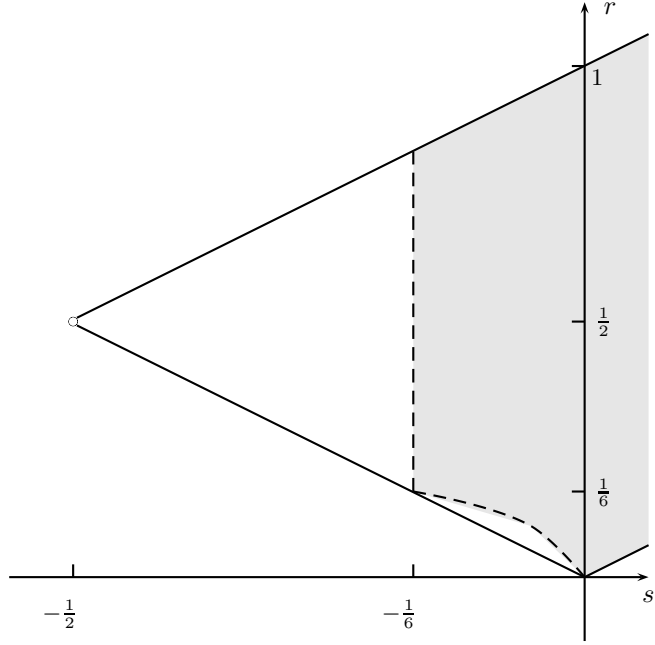
\begin{figure}\centering
\psset{xunit=1.7,yunit=1.7}
  \begin{pspicture}(-2.5,-2.5)(2.5,2.5)
\psline*[linecolor=gray!20](0.66,1.33)(2.5,2.25)(2.5,-1.75)(2,-2)(1.58,-1.6)(0.66,-1.33)(0.66,1.33)
    \psline{->}(-2.5,-2)(2.5,-2)
    \psline{->}(2,-2.5)(2,2.5)
    \psdot[dotstyle=o](-2,0)
    \psline(-1.97,0.03)(2.5,2.25)
    \psline(-1.97,-0.03)(2,-2)
    \psline(2,-2)(2.5,-1.75)
    \psline(1.9,2)(2,2) 
    \psline(1.9,0)(2,0) 
    \psline(1.9,-1.33)(2,-1.33) 
    \uput[u](2.1,1.75){\small $1$}
    \uput[u](2.15,-0.22){{\small $\frac{1}{2}$}}
    \uput[u](2.15,-1.55){\small$\frac{1}{6}$}
    \psline(-2,-2)(-2,-1.9) 
    \psline(0.66,-2)(0.66,-1.9) 
    \uput[u](-2.1,-2.5){{\small $-\frac{1}{2}$}}
    \uput[u](0.55,-2.5){\small$-\frac{1}{6}$}
    \uput[u](2.2,2.3){$r$}
    \uput[u](2.5,-2.3){$s$}
    \psline[linestyle=dashed](0.66,-1.33)(0.66,1.33)
    \pscurve[linestyle=dashed](0.66,-1.33)(1.58,-1.6)(2,-2)
  \end{pspicture}
\caption{ Global well-posedness holds in the shaded region by Theorem \ref{thm DKG in intro}. Local well-posedness holds inside the the lines $r=|s|$ and $r=s + 1$ for $s>\frac{-1}{2}$ by \cite{Machihara2010}.}
\end{figure}
The proof of Theorem \ref{thm DKG in intro} follows the argument used in \cite{Tesfahun2009} together with the bilinear estimates in Theorem \ref{thm main X^sb esimate}. More precisely, we use the $I$-method together with the induction on free waves approach of Selberg. The main idea, following the usual $I$-method, is to define a mild smoothing operator $I$ such that, firstly, for some large constant $N$, we have the estimate
    \begin{equation}\label{intro explanation of Imethod} \| If \|_{L^2(\RR)} \lesa N^{-s} \| f \|_{H^s(\RR)} \lesa N^{-s} \| I f \|_{L^2}.\end{equation}
  Secondly, we require $I$ to be the identity on low frequencies. We then try to estimate the growth of $\| I \psi(t) \|_{L^2}$ in terms of $t$.  It turns out that despite the fact that $I\psi$ no longer solves the DKG equation, there is sufficient cancelation of frequencies to ensure that the charge $\| I\psi(t) \|_{L^2_x}$ is almost conserved. This almost conservation property follows from the usual proof of the conservation of charge, together with a number of applications of Theorem \ref{thm main X^sb esimate}. Thus we can estimate the growth of $\| \psi(t) \|_{H^s}$ from (\ref{intro explanation of Imethod}). The induction on free waves approach of Selberg then allows us to control the scalar field $\phi$ and completes the proof of Theorem \ref{thm DKG in intro}.   \\

We now give a brief outline of this article. In Section \ref{sec - linear est}, we recall some properties of the $X^{s, b}$ and $H^{r, b}$ spaces which we require in the proof of Theorem \ref{thm DKG in intro}. The proof of Theorem \ref{thm DKG in intro} is contained in Section \ref{sec gwp for DKG}. In Section \ref{sec bilinear est} we prove that the conditions in Theorem \ref{thm main X^sb esimate} are sufficient for the estimate (\ref{thm main X^sb est - main trilinear est}). Finally, the counter examples showing that Theorem \ref{thm main X^sb esimate} is sharp up to equality are contained in Section \ref{sec counter examples}. \\

\textbf{Notation}:
The Fourier transform on $\RR$ of a function $f \in L^1(\RR)$ is denoted by $\widehat{f}(\xi) = \int_\RR f(x) e^{-i x\xi} dx$. We use the notation $\widetilde{f}(\tau, \xi)$ for the space-time Fourier transform of a function $f(t, x)$ on $\RR^{1+1}$. We write $a \lesa b$ if there is some constant $C$, independent of the variables under consideration, such that $a \les C b$. If we wish to make explicit that the constant $C$ depends on $\delta$ we write $a \lesa_\delta b$. Occasionally we write $a \ll b$ if $C<1$. We use $a \approx b$ to denote the inequalities $a \lesa b$ and $b \lesa a$.

All sums such as $\sum_{N} f(N)$ are over dyadic numbers $N \in 2^{\NN}$. Given dyadic variables $N_1, N_2, N_3 \in 2^{\NN}$, we use the short hand
        $$ \sum_{N_{max} \approx N_{med}} = \sum_{N_{max} \in 2^{\NN}} \sum_{\substack{  \\ N_{med} \in 2^{\NN} \\ N_{med} \approx N_{max} } } \sum_{\substack{ N_{min} \in 2^{\NN} \\ N_{min} \lesa N_{med} } }.$$
We let $\ind_\Omega$ denote the characteristic function of the set $\Omega$, we occasionally abuse notation and write $ \ind_{|x| \approx N} $
instead of $\ind_{\{ |x| \approx N\}}$. The standard Sobolev space $H^s$ is defined as the completion of $C^\infty_0$ using the norm
            $$ \| f \|_{H^s} = \| \al \xi \ar^s \widehat{f} \|_{L^2}.$$
If $u$ is a function of  $(t, x) \in \RR^{1+1}$ we use the notation
        $$ \| u[t] \|_{H^s} = \| u(t) \|_{H^s} + \| \p_t u(t) \|_{H^{s-1}}.$$
To handle solutions to the wave equation, we make use of the Banach space $\mathcal{H}^{r, b}$ defined via the norm
    $$ \| \q \|_{\mathcal{H}^{r, b}} = \| \q \|_{H^{r, b}} + \| \p_t \q \|_{H^{r-1, b}}.$$

The proof of Theorem \ref{thm DKG in intro} requires the use of the local in time versions of the $X^{s, b}_\pm$ and $\mathcal{H}^{r, b}$ spaces.  Let $S_{\Delta T} = [0, \Delta T] \times \RR$. We define $X^{s, b}_\pm(S_{\Delta T})$ by restricting elements of $X^{s,b}_\pm$ to $S_{\Delta T}$. More precisely,
    $$X^{s, b}_\pm(S_{\Delta T})= X^{s, b}_\pm/ \{ f \in X^{s, b}_\pm \, |\, f|_{S_{\Delta T}} = 0\}.$$
The local in time space $X^{s, b}_\pm(S_{\Delta T})$ is a Banach space with norm
       $$ \| \q \|_{X^{s,b}_\pm(S_{\Delta T})} = \inf_{ u = \q \text{ on } S_{\Delta T}} \| u \|_{X^{s, b}_\pm}.$$
If $b>\frac{1}{2}$, then we have the continuous embedding $X^{s, b}_\pm(S_{\Delta T}) \subset C\big([0, \Delta T], H^s\big)$. We define the Banach spaces $\mathcal{H}^{r, b}(S_{\Delta T})$ similarly and note that, if $b>\frac{1}{2}$, then we have the continuous embedding $\mathcal{H}^{r, b}(S_{\Delta T}) \subset C\big( [0, \Delta T], H^r \big) \cap C^1\big( [0, \Delta T], H^{r-1} \big)$.

\section{Linear Estimates}\label{sec - linear est}
Here we  briefly recall some of the important properties of the $X^{s, b}_\pm$ and $\mathcal{H}^{r, b}$ spaces which we make use of in the proof of Theorem \ref{thm DKG in intro}, for more details we refer the reader to \cite{D'Ancona2007b} and \cite{Tao2006b}. We start by recalling some properties of the localised spaces $X^{s, b}_\pm(S_{\Delta T})$.

\begin{lemma} \label{lem time dilation for Xsb}
Let $s \in \RR$, $0<\Delta T<1$, and $\nu \in C^\infty_0(\RR)$. If $- \frac{1}{2} < b_1 \les b_2 < \frac{1}{2}$ then
            $$ \Big\| \nu\Big( \frac{t}{\Delta T} \Big) u(t, x) \Big\|_{X^{s, b_1}_\pm} \lesa \Delta T^{b_2 - b_1 } \| u \|_{X^{s, b_2}_\pm}.$$
Consequently, we have $ \| u \|_{X^{s, b_1}_\pm(S_{\Delta T})} \lesa \Delta T^{b_2 - b_1} \| u \|_{X^{s, b_2}_\pm(S_{\Delta T})}$. Moreover if $-\frac{1}{2} <b < \frac{1}{2}$ then
        $$ \| \ind_{[0, \Delta T]}(t) u \|_{X^{s, b}_\pm} \lesa \| u \|_{X^{s, b}_\pm(S_{\Delta T})}$$
with constant independent of $\Delta T$.
\begin{proof}
   The first conclusion is well known and can be found in, for instance, \cite{Tao2006b}. The second conclusion is perhaps not as well known and for the convenience of the reader we include the proof here. The definition of $X^{s, b}_\pm(S_{\Delta T})$ together with a change of variables on the frequency side shows that is suffices to prove
        $$ \| \ind_{[0, \Delta T]}(t) f \|_{H^{b}} \lesa \| f \|_{H^{b} }.$$
  By duality we may assume that $0<b<\frac{1}{2}$. Then by a well-known characterisation of the Sobolev spaces $H^s$, (see for instance \cite{Adams2003}) we have
    \begin{align*}
      \| \ind_{[0, \Delta T]} f \|_{H^b}^2 &\approx \| \ind_{[0, \Delta T]} f \|_{L^2}^2  + \int_{\RR^2} \frac{ | \ind_{[0, \Delta T]}(t) f(t) - \ind_{[0, \Delta T]}(t') f(t') |^2}{|t - t'|^{1+ 2b}} dt dt'\\
            &\lesa \| f\|_{L^2}^2 + \int_0^{\Delta T} \int_0^{\Delta T} \frac{ |f(t) - f(t') |^2}{|t - t'|^{1+ 2b} } dt dt' + 2 \int_0^{\Delta T} \int_{t' \not \in [0, \Delta T]} \frac{|f(t)|^2}{ |t - t'|^{1+2b}} dt' dt\\
            &\lesa \| f\|_{H^b}^2 + 2 \int_0^{\Delta T} \int_{t' \not \in [0, \Delta T]} \frac{|f(t)|^2}{ |t - t'|^{1+2b}} dt' dt.
    \end{align*}
 To complete the proof we use Hardy's inequality (see for instance \cite[Lemma A.2]{Tao2006b}) together with the assumption $0<b<\frac{1}{2}$ to deduce that
    \begin{align*}
      \int_0^{\Delta T} \int_{t' \not \in [0, \Delta T]} \frac{|f(t) |^2}{ |t - t'|^{1+2b}} dt' dt
                &\lesa \int_0^{\Delta T} |f(t)|^2 \Big( \frac{1}{|t|^{2b}} + \frac{1}{|t - \Delta T|^{2b}} \Big) dt\\
                &\lesa \Big\| \frac{f(t)}{|t|^{b}} \Big\|^2_{L^2} + \Big\| \frac{f(t)}{|t-\Delta T|^{b}} \Big\|^2_{L^2}\\
                &\lesa \| f\|_{H^b}^2.
      \end{align*}

\end{proof}

\end{lemma}

To control the solution to the Dirac equation we make use of the energy estimate for the $X^{s, b}_\pm$ spaces.
\begin{lemma}\label{lem energy est for Xsb}
Let $s \in \RR$, $b>\frac{1}{2}$, and $0<\Delta T<1$. Suppose $f \in H^s$, $F \in X^{s, b-1}_{\pm}(S_{\Delta T})$,  and let $u$ be the solution to
    \begin{equation*}\begin{split}  \p_t u \pm \p_x u &= F \\
                                  u(0)&=f. \end{split} \end{equation*}
Then $u \in X^{s, b}_\pm(S_{\Delta T})$ and we have the estimate
        $$ \| u \|_{X^{s, b}_\pm(S_{\Delta T})} \lesa \| f\|_{H^s} + \| F \|_{X^{s, b-1}_\pm(S_{\Delta T})}. $$
\end{lemma}

We also require the $H^{r, b}$ versions of the above results.

\begin{lemma} \label{lem time dilation for Hrb}
Let $r \in \RR$, $0<\Delta T<1$, and $\nu \in C^\infty_0(\RR)$. Then if $- \frac{1}{2} < b_1 \les b_2 < \frac{1}{2}$ we have
            $$ \Big\| \nu\Big( \frac{t}{\Delta T} \Big) u(t, x) \Big\|_{H^{r, b_1}} \lesa \Delta T^{b_2 - b_1 } \| u \|_{H^{r, b_2}}.$$
Consequently, we have $ \| u \|_{H^{r, b_1}(S_{\Delta T})} \lesa \Delta T^{b_2 - b_1} \| u \|_{H^{r, b_2}(S_{\Delta T})}$.
\end{lemma}

\begin{lemma} \label{lem energy est for Hrb}
Let $r \in \RR$, $b>\frac{1}{2}$, $0<\Delta T<1$, and $m\in \RR$. Suppose $f \in H^r$, $g \in H^{r-1}$, and $F \in H^{r-1, b-1}(S_{\Delta T})$ and let $u$ be the solution to
    \begin{equation*}\begin{split} \Box u  &= m^2 u +  F \\
                                  u(0)&=f, \qquad \p_t u(0) = g. \end{split} \end{equation*}
Then $u \in \mathcal{H}^{r, b}(S_{\Delta T})$ and we have the estimate
        $$ \| u \|_{\mathcal{H}^{r, b}(S_{\Delta T})} \lesa \| f\|_{H^r} + \| g\|_{H^{r-1}} + \| F \|_{H^{r-1, b-1}(S_{\Delta T})}. $$
\begin{proof}
  See \cite{Tesfahun2009}.
\end{proof}
\end{lemma}

\section{Global Well-Posedness for the Dirac-Klein-Gordon Equation}\label{sec gwp for DKG}

We are now ready to consider the proof of global well-posedness for the DKG equation. To uncover the null structure for the DKG equation, we let $\psi =( \psi_+,  \psi_-)^T$.  Then the DKG equation (\ref{DKG general form}) can be written as
                \begin{equation}\label{DKG equation}
                \begin{split}
                    \p_t \psi_{\pm} \pm \p_x \psi_{\pm} &= -i M\psi_{\mp} + i \phi \psi_{\mp} \\
                                \Box \phi &= m^2 \phi - 2 \Re \big(\psi_+ \overline{\psi}_- \big)
                \end{split}
                \end{equation}
with initial data
            \begin{equation}\label{DKG data}    \psi_{\pm}(0) = f_{\pm} \in H^s, \qquad\phi(0) = \phi_0 \in H^r, \qquad \p_t \phi(0) = \phi_1\in H^{r-1}. \end{equation}
Note that the right hand side of (\ref{DKG equation}) has the bilinear product $\psi_+ \overline{\psi}_-$, which, as we have seen in Theorem \ref{thm main X^sb esimate}, behaves significantly better than the corresponding product with $++$. The $+-$ structure can also be seen in the term $\phi \psi_\pm$ via a duality argument \cite{Selberg2008}. These are the key observations used in the local well-posedness theory for the DKG equation.

To prove the global well-posedness result of Theorem \ref{thm DKG in intro}, by the local well-posedness result in \cite{Selberg2008}, it suffices to prove that the data norms $\| \psi_\pm(T) \|_{H^s}$, $\| u[T] \|_{H^r}$ remain finite for all large times $ 0<T<\infty$. To this end, we make use of the $I$-method together with ideas from \cite{Selberg2007} and \cite{Tesfahun2009}. Let $\rho_0 \in C^{\infty}$ be even, decreasing, and satisfy
        $$\rho_0(\xi) = \begin{cases} 1 \qquad &|\xi|<1\\
                                      |\xi|^s   &|\xi|>2. \end{cases} $$
Let $\rho(\xi) = \rho_0\Big(\frac{|\xi|}{N}\Big)$ and define the $I$ operator by $\widehat{I\psi}(\xi) = \rho(\xi) \widehat{\psi}(\xi).$
We have the following straightforward estimates. Firstly, since $s<0$, we have for any $\sigma \in \RR$,
            \begin{equation}\label{Imethod implies control of H^s norm} \| f\|_{H^\sigma} \lesa \| I f\|_{H^{\sigma - s}} \lesa N^{-s} \| f\|_{H^\sigma}. \end{equation}
In particular, by taking $\sigma=0$, we observe that to obtain control over $\| \psi(t) \|_{H^s_x}$, it suffices to estimate $\| I\psi(t)\|_{L^2_x}$. Secondly, if $\supp \widehat{g}\subset \{ |\xi| \gtrsim N \}$, $s<0$, and $s_1<s_2$, then we can trade regularity for decay in terms of $N$,
            \begin{equation}\label{Imethod can trade regularity for powers of N} \| g \|_{H^{s_1}} \lesa N^{s_1 - s_2} \| g \|_{H^{s_2}} \approx N^{s_1 - s_2 + s} \| Ig \|_{H^{s_2 - s}}. \end{equation}
Thirdly, we note that the $I$ operator is the identity on low frequencies, so if $\supp \widehat{f} \subset \{ |\xi| < N\}$ then $ If = f$. Finally, if $f$ is real-valued, then $If$ is also real-valued since $\rho$ was assumed to be even.

The $I$-method proceeds as follows. Assume we have a local solution
    $$\psi_{\pm} \in C\big([0, \Delta T], H^s\big), \qquad \phi \in C\big([0, \Delta T], H^r\big) \cap C^1\big( [0, \Delta T], H^{r-1}\big)$$
to (\ref{DKG equation}), (\ref{DKG data}). Note that from (\ref{Imethod implies control of H^s norm}) we have $I\psi(t) \in L^2_x$. We would like to use the conservation of charge to control $\| I\psi(t)\|_{L^2_x}$. However  $I\psi$ is no longer a solution to (\ref{DKG equation}) and so we can not expect $\| I \psi(t) \|_{L^2_x}$ to be conserved. Despite this, if we follow the proof of conservation of charge, then
    \begin{align}
        \p_t \int_{\RR}  |I \psi_+(t)|^2 + | I \psi_- (t)|^2 dx
                    &=2\Re\bigg( \int_\RR \overline{I\psi}_+ \p_t I \psi_+ + \overline{I\psi}_- \p_t I \psi_- dx \bigg)\notag \\
                    &=2\Re\bigg( \int_\RR \overline{I\psi}_+ \big( - \p_x I \psi_+ - i M  I \psi_- + i I(\phi \psi_-) \big) \notag \\
                    &\qquad \qquad \qquad \qquad + \overline{I\psi}_-  \big( \p_x I \psi_- - i M  I \psi_+ + i I(\phi \psi_+) \big) dx \bigg)\notag\\
                    &=2\Re\bigg(  i \int_\RR \overline{I\psi}_+ I(\phi \psi_-)  + \overline{I\psi}_- \p_t I(\phi \psi_+) dx \bigg). \label{derivation of almost conservation law}
        \end{align}
Now as $\phi$ is real-valued, $I^2\phi$ is also real-valued and hence
    $$ 2\Re\Big( i I^2 \phi \big( \overline{I \psi}_+ I \psi_- + \overline{I\psi}_- I \psi_+\big) \Big)=0.$$
Subtracting this term from (\ref{derivation of almost conservation law}) and using the fundamental theorem of Calculus then gives
    \begin{align}
        \sup_{t' \in [0, \Delta T] } \big( \| I \psi_+(t')  \|_{L^2_x}^2 + \| I \psi_-(t') &\|_{L^2_x}^2 \big)
        \les \| f_+ \|_{L^2}^2 + \| f_- \|_{L^2}^2 \notag\\
            & + 2 \sum_{\pm} \sup_{t' \in [0, \Delta T]} \Big| \int_0^{t'} \int_\RR \big( I(\phi \psi_\pm) - I^2 \phi I \psi_\pm\big) \overline{I \psi}_\mp dx dt\Big|.\label{growth of Ipsi} \end{align}
Thus provided we can show the last term in (\ref{growth of Ipsi}) is small, we can deduce that over a small time $[0, \Delta T]$, $\| I \psi_\pm(t)\|_{L^2}$ does not grow to large. The first step in this direction is the following.

\begin{lemma}\label{lem smoothing}
Let $\frac{-1}{4}<s<0$ and  $-s< r\les 1+2s$. Assume $b=\frac{1}{2}+\epsilon$ with $\epsilon>0$ sufficiently small. Then for any   $\Delta T \ll 1$, $N \gg 1$ we have
        \begin{align}
        \sup_{t' \in [0, \Delta T]} \Big| \int_0^{t'} \int_\RR \big(I(\phi &u) - I^2 \phi I u\big)\overline{I v} \,dx dt \Big| \notag\\
        &\lesa \Delta T^{\frac{1}{2} - 2\epsilon} N^{2s - r + 2\epsilon} \| I^2 \phi \|_{H^{r-2s, b}(S_{\Delta T})} \| Iu \|_{X^{0, b}_\pm (S_{\Delta T})} \| I v \|_{X^{0, b}_{\mp}(S_{\Delta T})}\label{lem smoothing - main eqn}\end{align}
where $S_{\Delta T} = [0, \Delta T] \times \RR$.
\begin{proof}
See Subsection \ref{subsec - proof of lem smoothing} below.
\end{proof}
\end{lemma}
\begin{remark}
The use of $I^2\phi$ instead of just $\phi$ or $I \phi$ on the right hand side of (\ref{lem smoothing - main eqn}) may require some explanation. Roughly speaking, the larger the negative exponent on $N$ in (\ref{lem smoothing - main eqn}), the better the eventual GWP result will be. Moreover, an examination of the proof of Lemma \ref{lem smoothing} shows that the exponent on $N$ depends entirely on the number of derivatives on $\phi$. In other words, we could replace the term $N^{2s - r} \| I^2\phi \|_{H^{r-2s, b}}$ with $N^{ks -r} \| I^k \phi \|_{H^{r-ks, b}}$ for any $k \in \NN$ (provided $r-ks \les 1$). However, the size of $\phi$ with respect to $N$ ends up being of the order $N^{-2s}$. This follows by observing that schematically $\phi$ is a solution to $\Box \phi = \psi^2$, and by (\ref{Imethod implies control of H^s norm}), the low frequency component of $\psi^2$ is essentially of size $N^{-2s}$. Thus it is natural to take $I^2 \phi$, which via (\ref{Imethod implies control of H^s norm}), also has size roughly $N^{-2s}$.
\end{remark}
\begin{remark}
The powers of $\Delta T$ and $N$ on the right hand side of (\ref{lem smoothing - main eqn}) are essentially sharp if we are working in the spaces $X^{s, b}_\pm$, $H^{s, b}$. This follows from the counter examples in Section \ref{sec counter examples} together with a scaling argument.
\end{remark}

Lemma \ref{lem smoothing} allows us to estimate the growth of $\| I \psi_\pm(t)\|_{L^2}$ on $[0, \Delta T]$, provided that we can control the size of the norms $\| I \psi_\pm\|_{X^{0, b}_\pm(S_{\Delta T} )}$ and $\| I^2 \phi \|_{H^{r-2s, b}(S_{\Delta T})}$.  This control is provided by a modification of the usual local well-posedness theory.

\begin{lemma}\label{lem mod lwp}
Let $\frac{-1}{6} < s<0$,  $-s<r \les \frac{1}{2} + 2s$, and $b= \frac{1}{2} + \epsilon$ with $\epsilon>0$ sufficiently small. Assume $f_\pm \in H^s$ and $\phi[0] \in H^r \times H^{r-1}$. Choose $\Delta T \ll 1$ and $N \gg 1$ such that
       \begin{equation} \label{lem mod lwp - required est for wave data} \Big( \Delta T^{\frac{1}{2} + r-2s - 3 \epsilon} + N^{-r+2s + 2\epsilon}\Big) \| I^2\phi[0] \|_{H^{r-2s}} \ll 1 \end{equation}
and
         \begin{equation} \label{lem mod lwp - required est for spinor data} \Big( \Delta T^{1 - \epsilon} + N^{  - \frac{1}{2}+ 2\epsilon}\Big) \Big(  \| If_+ \|_{L^2} + \| If_-\|_{L^2}\Big)^2 \ll 1. \end{equation}
Then the Dirac-Klein-Gordon equation (\ref{DKG equation}) with initial data (\ref{DKG data}) is locally well-posed on the domain $S_{\Delta T}=[0, \Delta T]\times \RR$. Moreover, the solution $(\psi, \phi)$ satisfies
            $$ \| I \psi_+ \|_{X^{0, b}_+(S_{\Delta T})} + \| I \psi_- \|_{X^{0, b}_-(S_{\Delta T})} \lesa \| If_+\|_{L^2} + \| If_-\|_{L^2}$$
and
            $$ \| I^2 \phi \|_{\mathcal{H}^{r-2s, b}(S_{\Delta T}) } \lesa \| I^2 \phi[0]\|_{H^{r-2s}} + \big( \| If_+\|_{L^2} + \| If_-\|_{L^2}\big)^2.$$
\begin{proof}
  See Subsection \ref{subsec - proof of lem mod lwp} below.
\end{proof}
\end{lemma}

\begin{remark}
  Note that since $\| I^2\phi[0]\|_{H^{r-2s}} \lesa N^{-2s}$, by choosing $N$ sufficiently large and $\Delta T$ sufficiently small, we can ensure that the inequality (\ref{lem mod lwp - required est for wave data}) is satisfied. A similar comment applies to (\ref{lem mod lwp - required est for spinor data}).
\end{remark}

\begin{remark}
The reason that we can extend the work of Tesfahun \cite{Tesfahun2009} is due to the conclusions in Lemma \ref{lem smoothing} and Lemma \ref{lem mod lwp}. In more detail, Lemma \ref{lem smoothing} improves \cite[Lemma 8]{Tesfahun2009} by adding a power of $\Delta T$ on the right hand side of (\ref{lem smoothing - main eqn}). Since $\Delta T$ will be taken small, this is a significant gain. Similarly, Lemma \ref{lem mod lwp} extends \cite[Theorem 8]{Tesfahun2009} by having a larger exponent on $\Delta T$ in (\ref{lem mod lwp - required est for wave data}). As a consequence, we can take $\Delta T$ larger, which improves the eventual GWP result. The point here is that the larger $\Delta T$ becomes, the fewer time steps of length $\Delta T$ are required to reach a large time $T$.
\end{remark}

We now follow the argument used in \cite{Tesfahun2009} and sketch the proof of Theorem \ref{thm DKG in intro}. The persistence of regularity result in \cite{Selberg2008} shows that it suffices to prove GWP in the case
     \begin{equation}\label{thm main X^sb esimate - proof - conditions on s, r}
        -\frac{1}{6}<s<0, \qquad s-\frac{1}{4} + \sqrt{\Big(s-\frac{1}{4}\Big)^2 - s}<r < \frac{1}{2}  + 2s. \end{equation}
Note that this region is non-empty as the intersection of the curves $s-\frac{1}{4} + \sqrt{\Big(s-\frac{1}{4}\Big)^2 - s}$ and $ \frac{1}{2} + 2s$ occurs at $s=-\frac{1}{6}$.

Choose some large time $T>0$ and assume $\epsilon>0$ is small. Let $N$ be some large fixed constant to be chosen later depending on the initial data $\|\psi(0)\|_{H^s}$ and $\|\phi[0]\|_{H^r}$, as well as the various constants appearing in Lemma \ref{lem smoothing} and Lemma \ref{lem mod lwp}. Take $\Delta T  = N^{ \frac{4 s - 2\epsilon }{ 1  + 2r-4s - 6\epsilon}}$. If $N$ is sufficiently large then from (\ref{Imethod implies control of H^s norm})
        \begin{align*}
            \Big( \Delta T^{\frac{1}{2} + r-2s - 3 \epsilon} + N^{-r+2s + 2\epsilon}\Big) \| I^2\phi[0] \|_{H^{r-2s}} &\ll 1\\
            \Big( \Delta T^{1 - \epsilon} + N^{ - \frac{1}{2} + 2 \epsilon}\Big) \Big(  \| If_+ \|_{L^2} + \| If_-\|_{L^2}\Big)^2 &\ll 1. \end{align*}
Therefore by Lemma \ref{lem mod lwp} we get a solution $(\psi, \phi)$ to (\ref{DKG equation}) on $[0, \Delta T]$. We would now like to repeat this argument $\frac{T}{\Delta T}$ times to advance to the time $T$. The only obstruction is the possible growth of the norms $\| I \psi_\pm(t) \|_{L^2}$ and $\| I^2\phi[t]\|_{H^{r-2s}}$. Our aim is to use Lemma (\ref{lem smoothing}) to show that $\| I \psi_\pm(t) \|_{L^2}$ is ``almost conserved'' and consequently obtain large time control over the norm $\| I \psi_\pm (t) \|_{L^2}$. This is accomplished by using an induction argument as follows.

Assume $n\lesa \frac{T}{\Delta T}$ and  suppose we have a solution $(\psi, \phi)$ on $[0, n\Delta T]$ with the bounds
            \begin{equation}\label{proof of gwp - induc assump for psi}
                \sup_{t \in [0, n \Delta T]} \Big( \|   I \psi_+(t)  \|_{L^2_x}^2 + \| I \psi_- (t) \|_{L^2_x}^2\Big)
                                \les 2 \| If_+ \|^2_{L^2_x} + 2\| If_-\|^2_{L^2_x} \end{equation}
and
            \begin{equation}\label{proof of gwp - induc assump for phi}
              \sup_{t \in [0, n \Delta T]} \|   I^2 \phi[t]  \|_{H^{r-2s}_x}
                                \les C^* \Big( \| I^2 \phi[0] \|_{H^{r-2s}_x} +   \big( \| If_+ \|_{L^2_x} + \| If_-\|_{L^2_x}\big)^2 \Big)
            \end{equation}
where the constant $C^*$ is some large constant independent of $N$, $\Delta T$, and $n$. If $N$ is sufficiently large, depending on $C^*$ and the initial data $\| f_\pm \|_{H^s}$, $\| \phi[0]\|_{H^r}$, then we can apply Lemma \ref{lem mod lwp} with initial data $\psi(n\Delta T)$, $\big(\phi(n\Delta T), \p_t \phi(n\Delta T) \big)$, and extend the solution to $[0, (n+1)\Delta T]$. Suppose we could show that the bounds (\ref{proof of gwp - induc assump for psi}) and (\ref{proof of gwp - induc assump for phi}) on $[0, n\Delta T]$ implied that they also hold on the larger interval $[0, (n+1) \Delta T]$ with the same constant $C^*$. Then by induction we would have (\ref{proof of gwp - induc assump for psi}) and (\ref{proof of gwp - induc assump for phi}) on $[0, T]$. Since $T$ was arbitrary, Theorem \ref{thm DKG in intro} would follow. Thus it suffices to verify the estimates (\ref{proof of gwp - induc assump for psi}) and (\ref{proof of gwp - induc assump for phi}) on the interval $[0, (n+1) \Delta T]$. We break this into two parts, proving the bound on $\| I \psi_\pm (t ) \|_{L^2}$,  and then estimating $\| I^2 \phi[t] \|_{H^{r-2s}}$. \\

\textbf{Bound on the Spinor $\psi_\pm$.} Let
        $$ \Gamma(z) = \sup_{t \in [0, z]} \Big( \|   I \psi_+(t)  \|_{L^2_x}^2 + \| I \psi_- (t) \|^2_{L^2_x}\Big) .$$
Note that the bounds (\ref{proof of gwp - induc assump for psi}) and (\ref{proof of gwp - induc assump for phi}) imply that
\begin{equation}\label{proof of gwp - rough bounds} \begin{split}
     \Gamma(n \Delta T)
                                &\les A N^{-2s} \\
     \sup_{t \in [0, n \Delta T]} \|   I^2 \phi[t]  \|_{H^{r-2s}_x}
                                &\les B N^{-2s}
     \end{split}
\end{equation}
where $A$ and $B$  depend on the initial data, the constant $C^*$, and $T$, but are independent of $n$, $N$, and $\Delta T$. If we now combine Lemma \ref{lem smoothing},  Lemma \ref{lem mod lwp} together with (\ref{growth of Ipsi}) we obtain the following control on the growth of $\Gamma(t)$.

\begin{corollary}[Almost conservation law]\label{cor almost conservation law}
  Let $\frac{-1}{6} < s <0$ and $ -s< r \les \frac{1}{2} + 2s$ and $b=\frac{1}{2} + \epsilon$ with $\epsilon>0$ sufficiently small. Suppose
            $$ \Delta T = N^{\frac{4 s - 2 \epsilon}{ 1  + 2r-4s - 6\epsilon} }$$
  and we have the bounds (\ref{proof of gwp - rough bounds}). Then provided $N$ is sufficiently large,
        $$ \Gamma(\Delta T) \les \Gamma(0) + C \Delta T^{\frac{1}{2} - 2\epsilon} N^{ - r + 2 \epsilon} \big( A +B\big) \Gamma(0).$$

  \begin{proof} By Lemma \ref{lem smoothing}, Lemma \ref{lem mod lwp}, and (\ref{growth of Ipsi}) it suffices to show that
                    $$ \Delta T^{\frac{1}{2} + r-2s - 3 \epsilon } N^{-2s} B + N^{-r + 2\epsilon} B \ll 1$$
    and
                    $$ \Delta T^{1-\epsilon} N^{-2s} A + N^{2\epsilon- \frac{1}{2} - 2s} B \ll 1.$$
    However these inequalities follow provided $\Delta T = N^{\frac{4 s - 2 \epsilon}{ 1  + 2r-4s - 6\epsilon} }$ and we choose $N$ sufficiently large.
  \end{proof}
\end{corollary}

We can now iterate the previous corollary to get control over $\Gamma(t)$ at time $(n+1)\Delta T$
        $$ \Gamma\big((n+1) \Delta T\big) \les \Gamma(0)  +  C n \Delta T^{\frac{1}{2} - 2\epsilon} N^{-r + 2\epsilon} (A + B) \Gamma(0).$$
Since the number of steps $n\lesa \frac{T}{\Delta T}$ we get
        $$ \Gamma\big((n+1) \Delta T\big)\les \Gamma(0)  +  C  T \Delta T^{-\frac{1}{2} - 2\epsilon} N^{-r + 2 \epsilon} (A + B) \Gamma(0).$$
We want to make the coefficient of the second term small. Thus we need to ensure that, using the requirement on $\Delta T$ in Corollary \ref{cor almost conservation law},
               \begin{equation}\label{almost conservation law deduction} C  T \Delta T^{-\frac{1}{2} - 2\epsilon} N^{-r + 2 \epsilon} (A + B) \approx N^{ \frac{ -(1+4\epsilon)(2s - \epsilon) }{1  +  2r - 4s - 6\epsilon}  -r + 2\epsilon} \ll 1. \end{equation}
By choosing $N$ large, and $\epsilon>0$ sufficiently small, we see that (\ref{almost conservation law deduction}) will follow provided $ -2s- r \big( 1 + 2r-4s\big) <0$.  Rearranging, we get the quadratic polynomial $2 r^2  +(1 - 4s) r + 2s >0$ and so we need
            $$  s - \frac{1}{4} + \sqrt{ \Big(s - \frac{1}{4}\Big)^2 - s} <r .$$
Therefore, provided we choose $N$ large enough, depending on $T$, $A$, and $B$, we get
            $$ \Gamma\big((n+1) \Delta T\big) \les 2\Gamma(0)$$
as required. \\

\textbf{Bound on $\phi$.} Recall that our goal was to show that, if the bounds (\ref{proof of gwp - induc assump for psi}) and (\ref{proof of gwp - induc assump for phi}) hold for $t \in [0, n \Delta T]$, then in fact they also held on the larger domain $[0, (n+1)\Delta T]$ (with the same constants). The bound for $\| I \psi_\pm\|_{L^2}$ was obtained above. Thus it remains to bound $\| I^2\phi[t] \|_{H^{r-2s}}$ on the interval $[0, (n+1)\Delta T]$. The argument that gives the required bound makes use of an idea due to Selberg in \cite{Selberg2007} on induction of free waves. The idea is to  break $\phi$ into a sum of homogeneous waves, together with an inhomogeneous term and then use an induction argument to estimate the contribution that each of these homogeneous waves makes to the size of $\| I^2\phi[t] \|_{H^{r-2s}}$. We note that this idea was also used in \cite{Tesfahun2009}. \\

We begin by observing that the induction assumptions (\ref{proof of gwp - induc assump for psi}) and (\ref{proof of gwp - induc assump for phi}) together with Lemma \ref{lem mod lwp} give for every $0 \les j \les n$
     \begin{equation}\label{bound on phi eqn1}    \|  I \psi_+ \|_{X^{0, b}_+(S_j)} + \|  I \psi_- \|_{X^{0, b}_-(S_j)} \les C_1 \Big(  \| I f_+ \|_{L^2_x} + \| If_-\|_{L^2_x}\Big) \end{equation}
where $S_j=[j \Delta T, (j+1)\Delta T]$ and the constant $C_1$ is independent of $C^*$, $j$, $n$,  $N$, and $\Delta T$. Suppose we could show that (\ref{bound on phi eqn1}) implies that
        \begin{equation}\label{bound on phi eqn2}
            \sup_{ t \in [n \Delta T, (n+1)\Delta T]} \| I^2 \phi[t] \|_{H^{r-2s}} \les C_2 \Big( \| I^2 \phi[0] \|_{H^{r-2s}_x} +  \Big( \| I f_+ \|_{L^2_x} + \| I f_- \|_{L^2_x}\Big)^2\Big).
        \end{equation}
Then by taking $C^*=C_2$ we see that the bound (\ref{proof of gwp - induc assump for phi}) holds for $t \in [0, (n+1) \Delta T]$. Thus by induction, together with the fact that the constants in (\ref{proof of gwp - induc assump for psi}) and (\ref{proof of gwp - induc assump for phi}) are independent of $n$,  we would obtain control over  the solution on $[0, T]$ and Theorem \ref{thm DKG in intro} would follow.

We now show that (\ref{bound on phi eqn1}) implies (\ref{bound on phi eqn2}). We make use of the following result which is a variant of a corresponding result in \cite{Tesfahun2009}.

\begin{lemma}\label{lem control of inhomogeneous term}
Let $m \in \RR$, $0<\Delta T<1$,  $\frac{-1}{4}<s<0$, $0<r<\frac{1}{2}  + 2s$, and $b>\frac{1}{2}$. Assume $u \in X^{s, b}_+(S_{\Delta T})$ and $v \in X^{s, b}_-(S_{\Delta T})$.  Then there exists a unique solution $\Phi \in H^{r, b}(S_{\Delta T})$ to
            \begin{align*}
                \Box \Phi &= \Re( u v) + m^2 \Phi \\
                      \Phi(0)&= \p_t \Phi(0) = 0.
            \end{align*}
on $S_{\Delta T}=[0, \Delta T]\times \RR$. Moreover we have
       \begin{equation} \label{lem control of inhomogeneous term - estimate} \sup_{ t \in [0, \Delta T]} \| I^2 \Phi[t] \|_{H^{r-2s}_x} \lesa \big( \Delta T + N^{ -\frac{1}{2} + 2\epsilon}\big) \| I u\|_{X^{0, b}_+(S_{\Delta T})} \| Iv \|_{X^{0, b}_-(S_{\Delta T})}. \end{equation}
\begin{proof}
The existence/uniqueness claim follows from Lemma \ref{lem energy est for Hrb} together with an application of Theorem \ref{thm main X^sb esimate}. To prove (\ref{lem control of inhomogeneous term - estimate}) we write $\Phi = \Phi_1 + \Phi_2$ where
        \begin{align*}
                \Box \Phi_1 &= \Re( u_{low} v_{low}) + m^2 \Phi_1\\
                      \Phi_1(0)&= 0, \qquad \p_t \Phi_1(0) = 0.
            \end{align*}
 and $\widehat{u_{low}} = \ind_{|\xi|< \frac{N}{2}} \widehat{u}$, $\widehat{v_{low}} = \ind_{|\xi|< \frac{N}{2}} \widehat{v}$. The standard representation of solutions to the Klein-Gordon equation, together with the Sobolev product law and the observation that $I^2(u_{low} v_{low}) = u_{low} v_{low}$, gives
    \begin{align*}
        \sup_{t \in [0, \Delta T]} \| I^2 \Phi_1[t] \|_{H^{r-2s}_x} &\lesa \int_0^{\Delta T} \| u_{low}(t)  v_{low}(t)  \|_{H^{r-2s -1}_x} dt\\
                       &\lesa \int_0^{\Delta T} \| u_{low}(t)\|_{L^2_x}  \|v_{low}(t)  \|_{L^2_x} dt \\
                       &\lesa \Delta T \| I u \|_{X^{0, b}_+(S_{\Delta T}) } \| I v \|_{X^{0, b}_-(S_{\Delta T})}.
    \end{align*}
 To bound the remaining term, $\Phi_2$, we note that by the energy estimate for $H^{s, b}$ spaces in Lemma \ref{lem energy est for Hrb},
        \begin{align}
               \sup_{t \in [0, \Delta T]} \| I^2 \Phi_2[t] \|_{H^{r-2s}_x} &\lesa \| I^2 \Phi_2 \|_{\mathcal{H}^{r - 2s, b}(S_{\Delta T})} \notag\\
                                    &\lesa \| I^2( u v - u_{low} v_{low}) \|_{H^{r - 2s -1, b-1}(S_{\Delta T})} \notag\\
                                    &\lesa \| u_{low} v_{hi} \|_{H^{-\frac{1}{2}, b-1}(S_{\Delta T})} + \| u_{hi} v_{low} \|_{H^{-\frac{1}{2}, b-1}(S_{\Delta T})} + \| u_{hi} v_{hi} \|_{H^{-\frac{1}{2}, b-1}(S_{\Delta T})} \label{lem control of inhomogeneous term - decomp into terms}
      \end{align}
 where $u_{hi} = u - u_{low}$ is the high frequency component of $u$, $v_{hi}$ is defined similarly, and we used the assumption $r < \frac{1}{2}  + 2s$. By Corollary \ref{cor wave-sobolev and X^sb estimate} we have the estimate
        \begin{equation}\label{lem control of inhomogeneous term - application of corollary}
            \|\psi_1 \psi_2 \|_{H^{ - \frac{1}{2}, b-1}} \lesa \| \psi_1 \|_{X^{-\frac{1}{2}- s_1 + 2\epsilon, b}_+} \| \psi_2 \|_{X^{s_1, b}_-}
        \end{equation}
 for $\frac{-1}{2} < s_1\les 0$. To control the first term in (\ref{lem control of inhomogeneous term - decomp into terms}) we use (\ref{lem control of inhomogeneous term - application of corollary}) with $s_1 = -\frac{1}{2} + 2\epsilon$ together with (\ref{Imethod can trade regularity for powers of N}) to obtain
 \begin{align*}
        \| u_{low} v_{hi} \|_{H^{-\frac{1}{2}, b-1}(S_{\Delta T})} &\lesa \| u_{low} \|_{X^{0, b}_+(S_{\Delta T})} \| v_{hi} \|_{X^{-\frac{1}{2} + 2 \epsilon, b}_-(S_{\Delta T})}\\
        &\lesa N^{-\frac{1}{2}  + 2\epsilon}  \| I u \|_{X^{0, b}_+(S_{\Delta T})} \| Iv \|_{X^{0, b}_-(S_{\Delta T})}
 \end{align*}
A similar application of (\ref{lem control of inhomogeneous term - application of corollary}) allows us to estimate the second term in (\ref{lem control of inhomogeneous term - decomp into terms}). Finally, for the last term in (\ref{lem control of inhomogeneous term - decomp into terms}) we use (\ref{Imethod can trade regularity for powers of N}) and (\ref{lem control of inhomogeneous term - application of corollary}) with $s_1= s$ to deduce that
    \begin{align*}
        \| u_{hi} v_{hi} \|_{H^{-\frac{1}{2}, b}(S_{\Delta T})} &\lesa \| u_{hi} \|_{X^{-\frac{1}{2} -s + 2\epsilon, b}_+(S_{\Delta T})} \| v_{hi} \|_{X^{s, b}_-(S_{\Delta T})}\\
        &\lesa N^{-\frac{1}{2}  + 2\epsilon}  \| I u \|_{X^{0, b}_+(S_{\Delta T})} \| Iv \|_{X^{0, b}_-(S_{\Delta T})}
 \end{align*}
 where we needed $-\frac{1}{2} - s + 2\epsilon\les s$ which holds provided $s> - \frac{1}{4}$ and $\epsilon$ sufficiently small.

\end{proof}
\end{lemma}

 We now have the necessary results to control the growth of $\| I^2\phi[t] \|_{H^{r-2s}}$. Let $0\les j\les n$ and define $\phi_j^{(0)}$ to be the solution to
        \begin{equation}
            \begin{split}
                \Box \phi_j^{(0)} &=m^2 \phi_j^{(0)}\\
                \phi_j^{(0)}(j\Delta T) &= \phi(j\Delta T), \qquad \p_t \phi_j^{(0)}(j\Delta T) = \p_t \phi(j\Delta T).
            \end{split}
        \end{equation}
 Let $\Phi_j = \phi - \phi_j^{(0)}$ be the inhomogeneous component of $\phi$.  The inequality (\ref{bound on phi eqn1}) together with Lemma \ref{lem control of inhomogeneous term} and the assumption $\Delta T = N^{\frac{4 s - 2 \epsilon}{1 + 2r - 4s - 6 \epsilon} } $, shows that for every $0\les j \les n$
         \begin{equation} \label{control of inhomogeneous term} \sup_{ t \in [j\Delta T, (j+1) \Delta T]} \| I^2 \Phi_j[t] \|_{H^{r-2s}_x} \lesa \Delta T \big(  \| I f_+ \|_{L^2_x} + \| If_-\|_{L^2_x} \big)^2. \end{equation}

We now claim that for $1\les j \les n$ we have the estimate
    \begin{equation}\label{control of KG component - estimate for free waves}
            \sup_{t\in [0, (n+1) \Delta T]} \| I^2 \phi^{(0)}_j[t] \|_{H_x^{r-2s}} \les \sup_{ t \in [0, (n+1) \Delta T]} \| I^2\phi^{(0)}_{j-1}[t] \|_{H^{r-2s}_x} +  C \Delta T\big( \| If_+ \|_{L_x^2} +  \| I f_- \|_{L^2_x}\big)^2.
\end{equation}
Assume for the moment that (\ref{control of KG component - estimate for free waves}) holds. Then after $n$ applications of (\ref{control of KG component - estimate for free waves}), together with the standard energy inequality for the homogeneous wave equation, we obtain
            \begin{align}
            \sup_{t\in [0, (n+1) \Delta T]} \| I^2 \phi^{(0)}_n[t] \|_{H^{r-2s}_x} &\les \sup_{t \in [0, (n+1) \Delta T]}  \| I^2 \phi^{(0)}_0[t] \|_{H^{r-2s}_x} + C n \Delta T \big( \| If_+ \|_{L^2_x} + \| I f_- \|_{L^2_x} \big)^2\notag \\
            &\lesa \| I^2\phi[0] \|_{H_x^{r-2s}} + C n \Delta T \big( \| If_+ \|_{L^2_x} + \| I f_- \|_{L^2_x} \big)^2.\label{control of KG component - estimate for free waves2}
            \end{align}
If we now combine (\ref{control of inhomogeneous term}) and (\ref{control of KG component - estimate for free waves2})  we see that since $n \lesa \frac{T}{\Delta T}$
  \begin{align*}
        \sup_{ t \in [n\Delta T, (n+1) \Delta T]} \| I^2\phi[t] \|_{H^{r-2s}_x} &\les \sup_{ t \in [n\Delta T, (n+1) \Delta T]} \| I^2\phi^{(0)}_{n}[t] \|_{H^{r-2s}_x} + \sup_{ t \in [n\Delta T, (n+1) \Delta T]} \| I^2\Phi_n[t] \|_{H^{r-2s}_x} \\
                            &\lesa \| I^2 \phi[0] \|_{H^{r-2s}_x} + (n+1) \Delta T \Big( \| I f_+ \|_{L^2_x} + \| I f_- \|_{L^2_x}\Big)^2 \\
                            &\lesa \| I^2 \phi[0] \|_{H^{r-2s}_x} +  \Big( \| I f_+ \|_{L^2_x} + \| I f_- \|_{L^2_x}\Big)^2
  \end{align*}
where the implied constant is independent of $N$, $C^*$, and $\Delta T$. Thus we obtain (\ref{bound on phi eqn2}) as required.

It only remains to prove (\ref{control of KG component - estimate for free waves}). We begin by observing that
        $$\big(\phi^{(0)}_{j} - \phi^{(0)}_{j-1} \big) (j \Delta T) = \phi( j\Delta T) - \phi^{(0)}_{j-1}( j\Delta T) = \Phi_{j-1}( j \Delta T).$$
Hence the difference $\phi^{(0)}_{j} - \phi^{(0)}_{j-1}$ satisfies the equation
        \begin{align*}
            \Box( \phi^{(0)}_{j} - \phi^{(0)}_{j-1} ) &= m^2( \phi^{(0)}_{j} - \phi^{(0)}_{j-1}) \\
                \big(\phi^{(0)}_{j} - \phi^{(0)}_{j-1}\big)(j\Delta T) &= \Phi_{j-1}(j\Delta T), \\
                 \p_t \big(\phi^{(0)}_{j} - \phi^{(0)}_{j-1}\big)(j\Delta T) &= \p_t \Phi_{j-1}(j\Delta T).
        \end{align*}
Therefore
        \begin{align*}
            \sup_{ t \in [0, (n+1)\Delta T]} \| I^2\phi^{(0)}_j[t] \|_{H^{r-2s}_x} &\les \sup_{ t\in [0, (n+1) \Delta T]} \| I^2 \phi^{(0)}_{j-1}[t] \|_{H^{r-2s}_x } +  \sup_{ t\in [0, (n+1) \Delta T]} \| I^2 \big(\phi^{(0)}_{j} - \phi^{(0)}_{j-1}\big)[t] \|_{H^{r-2s}_x } \\
            &\les \sup_{ t\in [0, (n+1) \Delta T]} \| I^2 \phi^{(0)}_{j-1}[t] \|_{H^{r-2s}_x } +  C \| \Phi_{j-1}[j\Delta T] \|_{H^{r-2s}_x}
        \end{align*}
and so (\ref{control of KG component - estimate for free waves}) follows from (\ref{control of inhomogeneous term}). Consequently, we deduce that the induction assumptions (\ref{proof of gwp - induc assump for psi}) and (\ref{proof of gwp - induc assump for phi}) hold on the larger interval $[0, (n+1) \Delta T]$ and hence Theorem \ref{thm DKG in intro} follows.

\subsection{Proof of Lemma \ref{lem smoothing}}\label{subsec - proof of lem smoothing}
Let $Q(f, g) = I(fg) - I^2 f I g$. Note that
        $$ \widehat{Q(f, g)}(\xi) = \int_\RR \big(\rho(\xi  ) - \rho(\xi - \eta)^2 \rho(\eta) \big) \widehat{f}(\xi - \eta) \widehat{g}(\eta) d\eta .$$
An application of Cauchy-Schwarz together with Lemma \ref{lem time dilation for Xsb} gives
        \begin{align*}
          \Big| \int_0^{t'} \int_\RR \big(I(\phi u) - I^2 \phi I u\big)\overline{I v} \,dx dt \Big| &\lesa \| \ind_{[0, t']} Q(\phi, u) \|_{X^{0, -\frac{1}{2} + \epsilon}_\mp} \| I v \|_{X^{0, \frac{1}{2} - \epsilon}_\mp([0, t']\times \RR)}\\
          &\lesa \|  Q(\phi, u) \|_{X^{0, -\frac{1}{2} + \epsilon}_\mp([0, t']\times \RR)} \| I v \|_{X^{0, \frac{1}{2} - \epsilon}_\mp([0, t']\times \RR)}\\
          &\lesa  \|Q(\phi, u) \|_{X^{0, -\frac{1}{2} + \epsilon}_\mp(S_{\Delta T})} \| I v \|_{X^{0, b}_\mp(S_{\Delta T})}.
        \end{align*}
Thus, by the definition of $X^{0, b}_\pm(S_{\Delta T})$, it suffices to prove that
        \begin{equation}\label{appendix lem smoothing - main eqn}
            \| Q(\phi, u) \|_{X^{0, -\frac{1}{2} +\epsilon}_\mp(S_{\Delta T})} \lesa \Delta T^{\frac{1}{2} - 2\epsilon} N^{2s - r + 2\epsilon} \|  I^2 \phi \|_{H^{r-2s, b}} \| I u \|_{X^{0, b}_\pm}.
        \end{equation}
where we may assume that $\phi$ and $u$ are supported in $[-\Delta T, 2 \Delta T]\times \RR$. Note that since the $I$ operator only acts on the spatial variable $x$, $I^2\phi$ and $I u$ are also supported in $[-\Delta T, 2 \Delta T]\times \RR$.
Write $\phi = \phi_{low} + \phi_{hi}$ and $u = u_{low} + u_{hi}$ where, as in the proof of Lemma \ref{lem control of inhomogeneous term}, we define $\widetilde{\phi}_{low} = \ind_{|\xi|\les \frac{N}{2}}\widetilde{\phi}$, and $u_{low}$ is defined similarly. We consider each of the possible interactions separately.  \\

\textbf{$\bullet$ Case 1 ( $low$-$low$).} In this case we simply note that $Q(\phi, u) =0$ and hence (\ref{appendix lem smoothing - main eqn}) holds trivially.\\

\textbf{$\bullet$ Case 2  ( $low$-$hi$).} We need to use the smoothing property of the bilinear form $Q(\phi, u)$ to transfer a derivative from $\phi_{low}$ to $u_{hi}$. More precisely, suppose $|\xi - \eta|<\frac{N}{2}$ and $|\eta|>\frac{N}{2}$. Then since $\rho'(z) \lesa N^{-s} |z|^{s-1}$ for $|z| \g \frac{N}{2}$ we have
    \begin{align*}
    |\rho(\xi) - \rho(\xi - \eta)^2 \rho(\eta)| &= |\rho(\xi) - \rho(\eta)| \\
                     &\lesa  N^{-s} |\eta|^{s-1} |\xi - \eta| \\
                     &\approx \rho(\eta) \frac{|\xi - \eta|}{|\eta|} \lesa \rho(\eta) \frac{ |\xi - \eta|^{r-2s}}{|\eta|^{r-2s}}
\end{align*}
 provided $r-2s < 1$. Hence
    $$|\widetilde{Q(\phi_{low}, u_{hi})}(\tau, \xi)|
    \lesa \int_{\RR^2} |\xi-\eta|^{r-2s}|\widetilde{\phi}_{low}(\tau - \lambda, \xi - \eta) |\eta|^{-r+2s} \rho(\eta) |\widetilde{u}_{hi}(\lambda, \eta)| d \lambda d \eta.   $$
 Thus we can move the derivative $|\nabla|^{r-2s}$ from $u_{hi}$ to $\phi_{low}$, where we let $\widehat{(|\nabla|^s f)}(\xi) = |\xi|^s \widehat{f}(\xi)$. This is the essential step which allows us to prove (\ref{appendix lem smoothing - main eqn}) in the $low$-$hi$ case. We now apply (\ref{Imethod can trade regularity for powers of N}) and Theorem \ref{thm main X^sb esimate} with $s_1 = s_2 = 0$,  $s_3 = 2\epsilon$, $b_1 = \frac{1}{2} -\epsilon$, $b_2=0$, and  $b_3 = b$ to obtain
 \begin{align*}
    \| Q(\phi_{low}, u_{hi} )\|_{X^{0, -\frac{1}{2} + \epsilon}_\mp(S_{\Delta T})} &\lesa \big\| |\nabla|^{r-2s} \phi_{low} |\nabla|^{-r+2s} Iu_{hi} \big\|_{X^{0, -\frac{1}{2} + \epsilon}_{\mp}}\notag\\
                    &\lesa \big\| |\nabla|^{r-2s} \phi_{low} \big\|_{L^2_{t, x} } \big\| |\nabla|^{-r + 2s } I u_{hi} \big\|_{X^{2\epsilon, b}_{\pm}}\notag\\
                    &\lesa \Delta T^{\frac{1}{2}} N^{-r +2s + 2\epsilon} \| I^2 \phi \|_{L^\infty_tH_x^{r-2s}} \| I u\|_{X^{0, b}_{\pm}} \\
                    &\lesa \Delta T^{\frac{1}{2}} N^{-r +2s + 2\epsilon} \| I^2 \phi \|_{H^{r-2s, b}} \| I u\|_{X^{0, b}_{\pm}}
    \end{align*}
 where  we used the assumption  $\supp \phi \subset \{ [-\Delta T, 2 \Delta T] \times \RR\}$. \\

\textbf{$\bullet$ Case 3 ( $hi$-$low$).} In this case we do not have to transfer any regularity and we simply use the estimate $ \rho(\xi) - \rho(\xi - \eta)^2 \rho(\eta) \lesa  1$. Then (\ref{Imethod can trade regularity for powers of N}) together with an identical application of Theorem \ref{thm main X^sb esimate} to the $low$-$hi$ case gives
   \begin{align*}  \| Q(\phi_{hi}, u_{low} )\|_{X^{0, -\frac{1}{2} + \epsilon}_\mp(S_{\Delta T})} &\lesa \big\| \phi_{hi}  u_{low} \big\|_{X^{0, -\frac{1}{2} + \epsilon}_{\mp}}\\
                    &\lesa \big\| \phi_{hi} \big\|_{L^2_{t, x} } \big\| u_{low} \big\|_{X^{2\epsilon, b}_{\pm}}\\
                    &\lesa \Delta T^{\frac{1}{2}} N^{2s -r  + 2\epsilon} \| I^2 \phi \|_{L^\infty_tH_x^{r-2s}} \| Iu \|_{X^{0, b}_{\pm}} \\
                    &\lesa \Delta T^{\frac{1}{2}}  N^{2s -r  + 2\epsilon} \| I^2 \phi \|_{H^{r-2s, b}} \| Iu \|_{X^{0, b}_{\pm}}
    \end{align*}
where as before, we used the assumption $\supp \phi \subset \{ [-\Delta T, 2 \Delta T] \times \RR\}$. \\

\textbf{$\bullet$ Case 4 ( $hi$-$hi$).} This is the most difficult case and we need to make full use of the generality of Theorem \ref{thm main X^sb esimate} to obtain the  term  $\Delta T^{\frac{1}{2} - \epsilon}$. We decompose $\phi_{hi} = \phi_{hi}^+ + \phi_{hi}^-$ where
    $$\widetilde{\phi}_{hi}^+ = \ind_{\{ \tau\xi <0\}} \widetilde{\phi}_{hi}$$
is the restriction of $\widetilde{\phi}_{hi}$ to the second and fourth quadrants of $\RR^{1+1}$. Note that $\| \phi^\pm \|_{X^{s, b}_\pm} \lesa \| \phi \|_{H^{s, b}}$. Assume that we have $\pm =+$, $\mp=-$ in (\ref{appendix lem smoothing - main eqn}), it will be clear that the proof will also apply to the $\pm=-$, $\mp =+$ case. \\

\textbf{$\bullet$ Case 4a ($hi$-$hi$ $+$).}
As in $hi$-$low$ case we start by discarding the smoothing multiplier $Q$. We now apply Theorem \ref{thm main X^sb esimate} with $s_1=-s + 2\epsilon$, $s_2 = s$, $s_3=0$, $b_1 = b_2 = \frac{1}{4}$, and $b_3 = \frac{1}{2} - \epsilon$ to obtain
    \begin{align*}
      \| Q( \phi_{hi}^+, u_{hi}) \|_{X^{0, -\frac{1}{2} + \epsilon}_-(S_{\Delta T})} &\lesa \| \phi^+_{hi} u_{hi} \|_{X^{0, -\frac{1}{2} + \epsilon}_-} \\
                                                    &\lesa \| \phi^+_{hi} \|_{X^{-s + 2\epsilon, \frac{1}{4}}_+} \| u_{hi} \|_{X^{s, \frac{1}{4}}_+} \\
                                                    &\lesa N^{2s-r + 2\epsilon} \| I^2 \phi \|_{H^{r-2s, \frac{1}{4} }} \| I u \|_{X^{0, \frac{1}{4}}_+}\\
                                                    &\lesa \Delta T^{\frac{1}{2} - \epsilon} N^{2s -r + 2\epsilon} \| I^2 \phi \|_{H^{r-2s, b}} \| I u \|_{X^{0, b}_+}
    \end{align*}
where we needed $-s<r$, $\epsilon>0$ sufficiently small, and in the final line we used the assumption that $\phi$, $u$, are compactly supported in the interval $[-\Delta T, 2\Delta T]$ together with Lemma \ref{lem time dilation for Xsb} and Lemma \ref{lem time dilation for Hrb}.\\

\textbf{$\bullet$ Case 4b  ( $hi$-$hi$ $-$).} Here we first apply Lemma \ref{lem time dilation for Xsb}, discard the multiplier $Q$, and then apply Theorem \ref{thm main X^sb esimate} with $s_1 =0$, $s_2=-s + \epsilon$, $s_3=s$, $b_1=b_2=\frac{1}{4}$, and $b_3=\frac{1}{2} + \epsilon$ to obtain
    \begin{align*}
      \| Q( \phi_{hi}^-, u_{hi}) \|_{X^{0, -\frac{1}{2} + \epsilon}_-(S_{\Delta T})} &\lesa \Delta T^{\frac{1}{4} - \epsilon}\| \phi^-_{hi} u_{hi} \|_{X^{0, -\frac{1}{4} }_-} \\
                                                    &\lesa \Delta T^{\frac{1}{4} - \epsilon}\| \phi^-_{hi} \|_{X^{-s + \epsilon, \frac{1}{4}}_-} \| u_{hi} \|_{X^{s, b}_+} \\
                                                    &\lesa \Delta T^{\frac{1}{4} - \epsilon}N^{2s-r + \epsilon} \| I^2 \phi \|_{H^{r-2s, \frac{1}{4} }} \| I u \|_{X^{0, b}_+}\\
                                                    &\lesa \Delta T^{\frac{1}{2} - 2\epsilon} N^{2s -r + \epsilon} \| I^2 \phi \|_{H^{r-2s, b}} \| I u \|_{X^{0, b}_+}
    \end{align*}
where, as previously, we used the assumption on the support of $\phi$ in the last line.

\subsection{Proof of Lemma \ref{lem mod lwp}}\label{subsec - proof of lem mod lwp}

Lemma \ref{lem mod lwp} follows by a standard fixed point argument using Lemma \ref{lem energy est for Xsb}, Lemma \ref{lem energy est for Hrb}, and the estimates
        \begin{equation}\label{appendix mod lwp - dirac est}
            \| I ( uv) \|_{X^{0, b-1}_{\pm}(S_{\Delta T})} \lesa \Big( \Delta T^{\frac{1}{2} + r - 2s - 3\epsilon} + N^{-r + 2 s + 2\epsilon}\Big) \| I^2 u \|_{H^{r-2s, b}(S_{\Delta T})} \| I v\|_{X^{0, b}_{\mp}(S_{\Delta T})} \end{equation}
and
        \begin{equation}\label{appendix mod lwp - wave est}
            \| I^2 ( uv) \|_{H^{r-2s -1, b-1}(S_{\Delta T})} \lesa \Big( \Delta T^{1 -\epsilon} + N^{-\frac{1}{2} + 2\epsilon}\Big) \| I u \|_{X_+^{0, b}(S_{\Delta T})} \| I v\|_{X^{0, b}_{-}(S_{\Delta T})}. \end{equation}
See for instance  \cite{Tesfahun2009}. \\

We start by proving (\ref{appendix mod lwp - dirac est}). As in the proof of Lemma \ref{lem smoothing}, we decompose $u= u_{low} + u_{hi}$ and $v=v_{low} + v_{hi}$. \\

\textbf{$\bullet$ Case 1 ( $low$-$low$).} We split $u_{low} = u_{low}^+ + u_{low}^{-}$ where we use the same notation as in Subsection \ref{subsec - proof of lem smoothing}, Case 4. Observe that an application of Theorem \ref{thm main X^sb esimate} gives
    \begin{equation} \label{lem mod lwp proof - low low case eqn 1} \int_{ \RR^{2}} \Pi_{j=1}^3 \psi_j dx dt \lesa \| \psi_1 \|_{X^{0, \epsilon}_\pm} \| \psi_2 \|_{X^{r-2s, \frac{1}{2} - r + 2s + \frac{\epsilon}{2} }_\pm} \| \psi_3 \|_{X^{0, \frac{1}{2}- \epsilon}_{\mp}}
    \end{equation}
provided that $0<r - 2s< \frac{1}{2}$ and $\epsilon>0$ is sufficiently small. Hence, using Lemma \ref{lem time dilation for Xsb} together with two applications of (\ref{lem mod lwp proof - low low case eqn 1}) we see that
    \begin{align*}
        \| I(u_{low} v_{low}) \|_{X^{0, b-1}_\pm(S_{\Delta T})} &\lesa \Delta T^{\frac{1}{2} - 2 \epsilon} \| u_{low}^\pm  v_{low} \|_{X^{0, - \epsilon}_\pm(S_{\Delta T})} + \| u_{low}^{\mp} v_{low} \|_{X^{0, b-1}_{\pm}(S_{\Delta T})} \\
                        &\lesa\Delta T^{\frac{1}{2} - 2 \epsilon} \| u_{low}^{\pm} \|_{X^{r-2s, \frac{1}{2} - r + 2s + \frac{\epsilon}{2}}_{\pm}(S_{\Delta T})} \| v_{low} \|_{X^{0, \frac{1}{2} + \epsilon}_{\mp}(S_{\Delta T})} \\
                        &\qquad \qquad \qquad + \| u_{low}^{\mp} \|_{X^{r-2s, \frac{1}{2} - r + 2s + \frac{\epsilon}{2}}_{\mp}(S_{\Delta T})} \| v_{low}  \|_{X^{0, \epsilon}_{\mp}(S_{\Delta T})}\\
                        &\lesa \Delta T^{\frac{1}{2} + r - 2s - 3\epsilon} \| I^2 u \|_{H^{r-2s, b}(S_{\Delta T})} \| I v \|_{X^{0, b}_\pm(S_{\Delta T})}.
    \end{align*}

\textbf{$\bullet$ Case 2 ( $low$-$hi$).} Note that Corollary \ref{cor wave-sobolev and X^sb estimate} implies that
        \begin{equation}\label{lem mod lwp proof - low hi case eqn3}
                    \| \psi \q \|_{X^{0, b-1}_{\pm}} \lesa \| \psi \|_{H^{s_1, b}} \| \psi \|_{X^{s_2, b}_{\mp}}
        \end{equation}
provided
    $$ s_1>0, \qquad s_2> - \frac{1}{2} + \epsilon, \qquad s_1 + s_2>\epsilon.$$
We now apply (\ref{lem mod lwp proof - low hi case eqn3}) with $s_1 = r-2s$, $s_2 = 2s - r + 2 \epsilon$ to get
        \begin{align*}
            \| I( u_{low} v_{hi} ) \|_{X^{0, b-1}_{\pm}(S_{\Delta T})} &\lesa \| u_{low} \|_{H^{r-2s, b}(S_{\Delta T})} \| v_{hi } \|_{X^{2s - r + 2\epsilon, b}_{\mp}(S_{\Delta T})} \\
            &\lesa N^{2s - r + 2\epsilon} \| I^2 u \|_{H^{r-2s, b}(S_{\Delta T})} \| I v \|_{X^{0, b}_{\mp}(S_{\Delta T})}.
        \end{align*}

\textbf{$\bullet$ Case 3 ( $hi$-$low$).} An application of (\ref{lem mod lwp proof - low hi case eqn3}) with $s_1 = 2\epsilon$, $s_2 = 0$ gives
    \begin{align*}
            \| I( u_{hi} v_{low} ) \|_{X^{0, b-1}_{\pm}(S_{\Delta T})} &\lesa \| u_{hi} \|_{H^{2\epsilon, b}(S_{\Delta T})} \| v_{low } \|_{X^{0, b}_{\mp}(S_{\Delta T})} \\
            &\lesa N^{2s - r + 2\epsilon} \| I^2 u \|_{H^{r-2s, b}(S_{\Delta T})} \| I v \|_{X^{0, b}_{\mp}(S_{\Delta T})}.
        \end{align*}

\textbf{$\bullet$ Case 4 ( $hi$-$hi$).} We apply (\ref{lem mod lwp proof - low hi case eqn3}) with $s_1=r$, $s_2=- r + 2\epsilon$ and observe that
    \begin{align*}
            \| I( u_{hi} v_{hi} ) \|_{X^{0, b-1}_{\pm}(S_{\Delta T})} &\lesa \| u_{hi} \|_{H^{r, b}(S_{\Delta T})} \| v_{hi } \|_{X^{ - r + 2\epsilon, b}_{\mp}(S_{\Delta T})} \\
            &\lesa N^{2s - r + 2\epsilon} \| I^2 u \|_{H^{r-2s, b}(S_{\Delta T})} \| I v \|_{X^{0, b}_{\mp}(S_{\Delta T})}
        \end{align*}
where we used the assumption $r>-s$ together with (\ref{Imethod can trade regularity for powers of N}).\\

We now prove prove (\ref{appendix mod lwp - wave est}). We again break $u = u_{low} + u_{hi}$ and $v = v_{low} + v_{hi}$ and consider each of the possible interactions separately. \\

\textbf{$\bullet$ Case 1 ( $low$-$low$).} Corollary \ref{cor wave-sobolev and X^sb estimate} together with the assumption $r - 2s <\frac{1}{2}$ gives
    \begin{align*}
        \| I^2(u_{low} v_{low} ) \|_{H^{r-2s -1, b-1}(S_{\Delta T})} &\lesa \| u_{low} v_{low}  \|_{H^{-\frac{1}{2} , b-1}(S_{\Delta T})} \\
                        &\lesa \| u_{low} \|_{X^{0, \epsilon}_+(S_{\Delta T})} \| v_{low} \|_{X^{0, \epsilon}_-(S_{\Delta T})} \\
                        &\lesa \Delta T^{1 - 2\epsilon} \| I u \|_{X^{0, b}_+(S_{\Delta T})} \| I v\|_{X^{0, b}_{-}(S_{\Delta T})}.
    \end{align*}

\textbf{$\bullet$ Case 2 ( $low$-$hi$).} For the remaining cases we will use the estimate
        \begin{equation}\label{lem mod lwp proof - low hi case eqn4}
                \| \psi \q \|_{H^{-\frac{1}{2}, b-1}} \lesa \| \psi\|_{X^{s_1, b}_+} \| \q\|_{X^{s_2, b}_-}
        \end{equation}
which follows from Corollary \ref{cor wave-sobolev and X^sb estimate} provided
        $$ s_1>-\frac{1}{2},\qquad s_2>-\frac{1}{2}, \qquad s_1 + s_2> -\frac{1}{2} + \epsilon.$$
The $low$-$hi$ case now follows by taking $s_1=0$, $s_2 =  - \frac{1}{2} + 2\epsilon$ and observing that
        \begin{align*}
                  \| I^2 ( u_{low} v_{hi} ) \|_{H^{r-2s -1, b-1}(S_{\Delta T})} &\lesa \| u_{low} v_{hi}  \|_{H^{-\frac{1}{2} , b-1}(S_{\Delta T})} \\
                  &\lesa      \|  u_{low}  \|_{X_+^{0, b}(S_{\Delta T})} \|  v_{hi} \|_{X^{-\frac{1}{2} + 2\epsilon, b}_{-}(S_{\Delta T})}\\
                  &\lesa  N^{-\frac{1}{2} + 2\epsilon} \| I u \|_{X_+^{0, b}(S_{\Delta T})} \| I v\|_{X^{0, b}_{-}(S_{\Delta T})}.
        \end{align*}

\textbf{$\bullet$ Case 3 ( $hi$-$low$).} Follows by taking $s_1 = -\frac{1}{2} + 2\epsilon$, $s_2= 0$ in (\ref{lem mod lwp proof - low hi case eqn4}) and using an identical argument to the previous case. \\

\textbf{$\bullet$ Case 4 ( $hi$-$hi$).} As before, we use (\ref{lem mod lwp proof - low hi case eqn4}) with $s_1 = - \frac{1}{2} +2\epsilon - s$ and $s_2 = s$ and apply a similar argument to the above cases.

\section{Bilinear Estimates}\label{sec bilinear est}

In this section we prove Theorem \ref{thm main X^sb esimate}. To help simplify the proof, we start by introducing some notation. Let $m : \RR^3\times\RR^3 \rightarrow \CC$ and consider the inequality
        \begin{equation}\label{general multiplier estimate}
            \Big| \int_{\Gamma} m(\tau, \xi) \Pi_{j=1}^3 f_j(\tau_j, \xi_j) d\sigma(\tau, \xi)\Big| \lesa \Pi_{j=1}^3 \| f_j \|_{L^2_{\tau, \xi}}
        \end{equation}
where $\tau, \xi \in \RR^3$, $\Gamma = \{ \xi_1 + \xi_2 + \xi_3 =0, \,\,\,\, \tau_1 + \tau_2 +\tau_3 =0\}$, and  $d\sigma$ is the surface measure on the hypersurface $\Gamma$. Without loss of generality, we may assume $f_j\g 0$ as we are using $L^2$ norms on the right hand side of (\ref{general multiplier estimate}). Note that the $X^{s, b}$ estimate contained in Theorem \ref{thm main X^sb esimate} can be written in the form (\ref{general multiplier estimate}) after applying Plancherel and relabeling.

Following Tao in \cite{Tao2001}, for a multiplier $m$,  we use the notation $\| m \|_{[3, \RR\times \RR]}$ to denote the optimal constant in (\ref{general multiplier estimate}). This norm $\| \cdot\|_{[3, \RR \times \RR]}$ was studied in detail in \cite{Tao2001}. We recall the following elementary properties. Firstly, if $m_1 \les m_2$ then it is easy to see that $\| m_1 \|_{[3, \RR\times \RR]} \les \| m_2 \|_{[3, \RR \times \RR]}$. Secondly, via Cauchy-Schwarz, for $j, k  \in \{1, 2, 3\}$, $j \neq k$, we have the characteristic function estimate
        \begin{equation}\label{characteristic function est}
            \| \ind_A(\tau_j, \xi_j) \ind_B(\tau_k, \xi_k) \|_{[3, \RR\times \RR]} \lesa \sup_{(\tau, \xi) \in \RR^2} \big| \{ (\lambda, \eta) \in A \,\,:\,\,(\tau - \tau, \xi - \xi) \in B\, \}\big|^{\frac{1}{2}}
        \end{equation}
where $|\Omega|$ denotes the measure of the set $\Omega \subset \RR^2$. We refer the reader to \cite{Tao2001} for a proof as well a number of other properties of the norm $\| \cdot \|_{[3, \RR \times \RR]}$.

Let $$ \lambda_1 = \tau_1 \pm \xi_1, \qquad \lambda_2 = \tau_2 \pm  \xi_2, \qquad \lambda_3 = \tau_3 \mp \xi_3.$$
Note that if $(\tau, \xi ) \in \Gamma$, then
            \begin{equation}\label{lambda sum} \lambda_1 + \lambda_2 + \lambda_3 = \pm2 \xi_3. \end{equation}
Let $N_j, L_j \in 2^{\NN}$, $j = 1, 2, 3$, be dyadic numbers. Our aim is to decompose the $\xi_j$ and $\lambda_j$ variables dyadically, and reduce the problem of estimating $\| m \|_{[3, \RR \times \RR]}$ to trying to bound the frequency localised version
        $$ \Big\| m(\tau, \xi) \Pi_{j=1}^3 \ind_{\{|\xi_j| \approx N_j, \,|\lambda_j|\approx L_j\}} \Big\|_{[3, \RR \times \RR]}$$
together with computing a dyadic summation. Note that if we restrict $|\xi_j| \approx N_j$, then since $\xi_1+ \xi_2 + \xi_3 =0$ we must have $N_{max} \approx N_{med}$ where, as in the introduction, $N_{max} = \max\{ N_1, N_2, N_3\}$, $N_{med}$ and $N_{min}$ are defined similarly. Similarly, if $|\lambda_j| \approx L_j$, then (\ref{lambda sum}) implies that $L_{max} \approx \max\{ L_{med}, N_3\}$.  Hence
    $$ 1 \approx \sum_{N_{max} \approx N_{med} } \sum_{L_{max} \approx \max\{ N_3, L_{med} \} }\Pi_{j=1}^3 \ind_{\{|\xi_j| \approx N_j, \,|\lambda_j|\approx L_j\}}.$$
Combining these observations with results from \cite{Tao2001} leads to the following.

\begin{lemma}\label{lem reduction to dyadic pieces}
  $$\| m \|_{[3, \RR \times \RR]}  \lesa \sup_{N} \sum_{N_{max} \approx N_{med} \approx N} \sum_{L_{max} \approx \max\{ N_3, L_{med} \} } \Big\| m(\tau, \xi) \Pi_{j=1}^3 \ind_{\{|\xi_j| \approx N_j, \,|\lambda_j|\approx L_j\}} \Big\|_{[3, \RR \times \RR]}.$$
\begin{proof}
The inequality follows from the triangle inequality together with \cite[Lemma 3.11]{Tao2001}. Alternatively, we can just compute by hand. For ease of notation, let $a_{N_1} = \| f_1 \ind_{|\xi_1| \approx N_1} \|_{L^2}$, $b_{N_2} = \| f_2 \ind_{|\xi_2| \approx N_2} \|_{L^2}$, $c_{N_3} = \| f_3 \ind_{|\xi_3| \approx N_3} \|_{L^2}$,  and $A_{N_1, N_2, N_3} = \big\| m(\tau, \xi) \Pi_{j=1}^3 \ind_{|\xi_j| \approx N_j} \big\|_{[3, \RR\times \RR]}$. Then since $\xi_j$ lie on the surface  $\Gamma$, we have  $\xi_1 + \xi_2 + \xi_3 = 0$  and so
      \begin{align*} \int_{\Gamma} m(\tau, \xi) \Pi_{j=1}^3 f_j(\tau_j, \xi_j) d\sigma(\tau, \xi) &= \sum_{N_{max} \approx N_{med} } \sum_{N_{min}
      \les N_{med}} \int_{\Gamma} m(\tau, \xi) \Pi_{j=1}^3 f_j(\tau_j, \xi_j) \ind_{|\xi_j| \approx N_j} d\sigma(\tau, \xi) \\
      &\les \sum_{N_{max} \approx       N_{med}} \sum_{N_{min} \lesa N_{med}} a_{N_1} b_{N_2} c_{N_3} A_{N_1, N_2, N_3}.
      \end{align*}
Without loss of generality we may assume that $N_1 \g N_2 \g N_3$ and so $N_1 \approx N_2$. For simplicity we also assume that $N_1=N_2$ as the general case $N_1 \approx N_2$ is essentially the same. Then
    \begin{align*} \int_{\Gamma} m(\tau, \xi) \Pi_{j=1}^3 f_j(\tau_j, \xi_j) d\sigma(\tau, \xi) &\les \sum_{N_1} a_{N_1} b_{N_1} \sum_{N_3 \les N_1} c_{N_3} A_{N_1, N_1, N_3} \\
    &\lesa \Big( \sup_{N_3} c_{N_3}\Big) \Big(  \sup_{N_1} \sum_{N_3\les N_1} A_{N_1, N_1, N_3} \Big) \sum_{N_1} a_{N_1} b_{N_1} \\
    &\lesa \Big(  \sup_{N_1} \sum_{N_3\les N_1} A_{N_1, N_1, N_3} \Big) \Pi_{j=1}^3 \| f_j \|_{L^2}.
\end{align*}
Thus we have
    $$\| m \|_{[3, \RR \times \RR]}  \lesa \sup_{N} \sum_{N_{max} \approx N_{med} \approx N} \sum_{N_{min}\les N_{med}} \Big\| m(\tau, \xi) \Pi_{j=1}^3 \ind_{\{|\xi_j| \approx N_j \}} \Big\|_{[3, \RR \times \RR]}.$$
To decompose the $\lambda_j$ variables follows an similar argument. We omit the details.

\end{proof}
\end{lemma}

We now come to the proof of Theorem \ref{thm main X^sb esimate}. To begin with, by taking the Fourier transform and relabeling, the required estimate (\ref{thm main X^sb est - main trilinear est}) is equivalent to showing
\begin{equation}\label{main est fourier side}
            \Big| \int_{\Gamma} \mathfrak{m}(\tau, \xi) \Pi_{j=1}^3 f_j(\tau_j, \xi_j) d\sigma(\tau, \xi)\Big| \lesa \Pi_{j=1}^3 \| f_j \|_{L^2_{\tau, \xi}}
    \end{equation}
where
    $$\mathfrak{m}(\tau, \xi) = \frac{ \al \xi_1 \ar^{-s_1} \al \xi_2 \ar^{-s_2} \al \xi_3 \ar^{-s_3}}{\al\tau_1 \pm \xi_1 \ar^{b_1} \al \tau_2 \pm \xi_2 \ar^{b_2} \al \tau_3 \mp \xi_3 \ar^{b_3} }. $$
Note that Theorem \ref{thm main X^sb esimate} follows from the estimate $\| \mathfrak{m} \|_{[3, \RR \times \RR]} < \infty$. Now since
        $$ \big\| \mathfrak{m} \,\Pi_{j=1}^3 \ind_{\{|\xi_j| \approx N_j, \,|\lambda_j|\approx L_j\}} \big\|_{[3, \RR\times \RR]} \approx \big\|\Pi_{j=1}^3 \ind_{\{|\xi_j| \approx N_j, \,|\lambda_j|\approx L_j\}} \big\|_{[3, \RR\times \RR]} \Pi_{j=1}^3 N_j^{-s_j} L_j^{-b_j},$$
an application of Lemma \ref{lem reduction to dyadic pieces} shows that is suffices to estimate, for every $N \in 2^{\NN}$,
             \begin{align}\label{main dyadic sum}
                 \sum_{N_{max} \approx N_{med} \approx N} N_1^{-s_1} N_2^{-s_2} N_3^{-s_3} \sum_{L_{max} \approx \max\{ L_{med}, N_3\} } L_1^{-b_1} L_2^{-b_2} L_3^{-b_3} \big\|  \Pi_{j=1}^3 \ind_{\{|\xi_j|\approx N_j, \, |\lambda_j| \approx L_j \}}\|_{[3, \RR \times \RR]}. \end{align}
The first step to estimate this sum is the following estimate on the size of the frequency localised multiplier.
\begin{lemma}\label{lem dyadic multiplier estimate}
   $$ \big\|  \Pi_{j=1}^3 \ind_{\{|\xi_j|\approx N_j, \, |\lambda_j| \approx L_j \}}\|_{[3, \RR \times \RR]} \lesa \min\Big\{ N_{min}^{\frac{1}{2}} L_{min}^{\frac{1}{2}}, \,L_1^{\frac{1}{2}} L_3^{\frac{1}{2}}, \,L_2^{\frac{1}{2}}L_3^{\frac{1}{2}}\Big\} $$
\begin{proof}
Let $I= \big\|  \Pi_{j=1}^3 \ind_{\{|\xi_j|\approx N_j, \, |\lambda_j| \approx L_j \}}\|_{[3, \RR \times \RR]} $.  If we  let $A= \ind_{|\lambda_j|\approx L_j, \,\, |\xi_j |\approx N_j}$ and $B = \ind_{|\lambda_k| \approx L_k, \, |\xi_k| \approx N_k}$ in (\ref{characteristic function est}), then an application of Fubini gives
    \begin{align*}
        I       &\lesa \big\|   \ind_{\{|\xi_j|\approx N_j, \, |\lambda_j| \approx L_j \}} \ind_{\{|\xi_k|\approx N_k, \, |\lambda_k| \approx L_k \}} \|_{[3, \RR \times \RR]} \\
            &\lesa \sup_{\lambda, \xi \in \RR} \big| \big\{ |\lambda_j| \approx L_j \,\, : \,\, |\lambda - \lambda_j| \approx L_k \big\} \big|^{\frac{1}{2}} \big| \big\{ |\xi_j| \approx N_j \,\, : \,\, |\xi - \xi_j| \approx N_k \, \big\} \big|^{\frac{1}{2}}\\
            &\lesa \min\{ L_j^{\frac{1}{2}}, L_k^{\frac{1}{2}}\} \min\{ N_j^{\frac{1}{2}}, N_k^{\frac{1}{2}} \}
        \end{align*}
and hence $I \lesa L_{min}^{\frac{1}{2}} N_{min}^{\frac{1}{2}}$. On the other hand, another application of (\ref{characteristic function est}) together with a change of variables gives
    \begin{align*}
      I \lesa \big\|   \ind_{\{|\lambda_1| \approx L_1 \}} \ind_{\{ |\lambda_3| \approx L_3 \}} \|_{[3, \RR \times \RR]}
        &\lesa \sup_{\tau, \xi \in \RR} \big| \big\{ |\tau_1 \pm \xi_1| \approx L_1\,\, : \, |\tau \mp\xi - (\tau_1 \mp \xi_1) | \approx L_3 \big\} \big|^{\frac{1}{2}}\\
        &\lesa L_1^{\frac{1}{2}} L_3^{\frac{1}{2}}.
    \end{align*}
A similar argument gives $I \lesa L_2^{\frac{1}{2}} L_3^{\frac{1}{2}}$ and hence lemma follows.
\end{proof}
\end{lemma}

We are now ready to preform the computations needed to estimate the dyadic summation (\ref{main dyadic sum}). We split this into two parts, by computing the inner summation and then the outer summation. We note the following estimate
        $$ \sum_{a\les N \les b} N^{\delta} \approx \begin{cases}
                                a^{\delta} \qquad \qquad&\delta<0 \\
                                \log(b)             &\delta=0\\
                                b^{\delta}          &\delta>0 \end{cases}$$
which we use repeatedly. Moreover, we have $\log( r ) \lesa r^\epsilon$ for any $\epsilon>0$ and $r\g 1$.

\begin{lemma}\label{lem inner summation}
Let $b_j + b_k>0$ and $b_1 + b_2 + b_3>\frac{1}{2}$. Then for any sufficiently small $\epsilon>0$
       \begin{align*} \sum_{L_{max} \approx \max\{ L_{med}, N_3\}}& L^{-b_1}_1 L^{-b_2}_2 L^{-b_3}_3 \big\|  \Pi_{j=1}^3 \ind_{\{|\xi_j|\approx N_j, \, |\lambda_j| \approx L_j \}}\|_{[3, \RR \times \RR]}\\
                &\lesa N_3^\epsilon \,\Big(
                    N_3^{\frac{1}{2} - b_1 - b_2 - b_3} N_{min}^\frac{1}{2} + N_{3}^{-b_3} N_{min}^{\frac{1}{2}} + N_3^{-b_{min}} N_{min}^{(\frac{1}{2} - b_{max})_+ +\, (\frac{1}{2} - b_{med})_+}\Big).
                    \end{align*}
\begin{proof}
We split into the cases $L_{med} \les N_3 $ and $L_{med}\g N_3$. \\

\textbf{ $\bullet$ Case 1 ($L_{med} \les N_3$).} Since the the righthand side of Lemma \ref{lem dyadic multiplier estimate} does not behave symmetrically with respect to the sizes of the $L_j$, we need to decompose further into $L_{max} =L_3$ and $L_{max} \neq L_3$. \\

 \textbf{$\bullet$ Case 1a ($L_{med} \les N_3$ and $L_{max} \neq L_3$).} We have by Lemma \ref{lem dyadic multiplier estimate}
    $$ \big\|  \Pi_{j=1}^3 \ind_{\{|\xi_j|\approx N_j, \, |\lambda_j| \approx L_j \}}\big\|_{[3, \RR \times \RR]}
                        \lesa L_{min}^{\frac{1}{2}} \min\{ N_{min}^{\frac{1}{2}}, L_{med}^{\frac{1}{2}}\}.$$
 Since the righthand side is symmetric under permutations of $\{1, 2, 3\}$, we may assume $L_1 \g L_2 \g L_3$. Then for any $\epsilon>0$
     \begin{align}
      \sum_{L_{max} \approx N_3 \gtrsim L_{med} } L^{-b_1}_1 L^{-b_2}_2 L^{-b_3}_3 \big\|  \Pi_{j=1}^3 &\ind_{\{|\xi_j|\approx N_j, \, |\lambda_j| \approx L_j \}}\big\|_{[3, \RR \times \RR]} \notag \\
      &\lesa N^{-b_1}_3 \sum_{L_2 \les N_3} L^{-b_2}_2 \min\{ N_{min}^{\frac{1}{2}}, L_2^{\frac{1}{2}}\} \sum_{L_3 \les L_2} L_3^{\frac{1}{2}-b_3}\notag\\
      &\lesa N^{-b_1}_3 \sum_{L_2 \les N_3} L^{(\frac{1}{2}- b_3)_+ - b_2}_2 \log(L_2) \min\{ N_{min}^{\frac{1}{2}}, L_2^{\frac{1}{2}}\} \notag\\
      &\lesa N^{-b_1 + \frac{\epsilon}{2} }_3 \sum_{L_2 \les N_{min}}  L_2^{(\frac{1}{2} - b_3)_+ \,+\, \frac{1}{2} - b_2}  \notag\\
      &\qquad\qquad + N_{min}^{\frac{1}{2}} N^{-b_1 + \frac{\epsilon}{2} }_3 \sum_{N_{min} \les L_2 \les N_3} L^{(\frac{1}{2} - b_3)_+ - b_2}_2   \label{lem dyadic multiplier estimate - eqn 4}
        \end{align}
Now for the first sum in (\ref{lem dyadic multiplier estimate - eqn 4}) we have
    \begin{align*}
         N^{-b_1}_3 \sum_{L_2 \les N_{min}}  L_2^{(\frac{1}{2} - b_3)_+ \,+\, \frac{1}{2} - b_2}
                &\lesa  N_{min}^{ \big((\frac{1}{2} - b_3)_+ + \frac{1}{2} - b_2\big)_+ } N_3^{ - b_1} \log(N_{min})\\
                &\lesa   N_{min}^{ (\frac{1}{2} - b_{max})_+ +( \frac{1}{2} - b_{med} )_+ } N_3^{ - b_{min} + \frac{\epsilon}{2} }.
    \end{align*}
For the second sum we first consider the case $(\frac{1}{2} - b_3)_+ - b_2>0$. Then
             \begin{align*}
               N_{min}^{\frac{1}{2}} N^{-b_1}_3 \sum_{N_{min} \les L_2 \les N_3} L^{(\frac{1}{2} - b_3)_+ - b_2}_2
                    &\lesa  N^{\frac{1}{2}}_{min} N_3^{(\frac{1}{2} - b_3)_+ -b_1 - b_2  } \\
                    &\lesa N^{\frac{1}{2}}_{min} N_3^{(\frac{1}{2} - b_{max})_+ - b_{med} - b_{min}}
             \end{align*}
On the other hand if $(\frac{1}{2} - b_3)_+ - b_2\les 0$ we get
        \begin{align*}
               N_{min}^{\frac{1}{2}} N^{-b_1}_3 \sum_{N_{min} \les L_2 \les N_3} L^{(\frac{1}{2} - b_3)_+ - b_2}_2
                        &\lesa  N^{\frac{1}{2} - b_2 + (\frac{1}{2}- b_3)_+}_{min} N_3^{-b_1} \log(N_3) \\
                        &\lesa N^{(\frac{1}{2} - b_{max})_+ + (\frac{1}{2} - b_{med})_+ }_{min} N^{-b_{min}+\frac{\epsilon}{2}}_3. \end{align*}
Together with (\ref{lem dyadic multiplier estimate - eqn 4}) this then gives
    \begin{align*}
        \sum_{L_{max} \approx N_3 \gtrsim L_{med} } L^{-b_1}_1 L^{-b_2}_2 &L^{-b_3}_3 \big\|  \Pi_{j=1}^3 \ind_{\{|\xi_j|\approx N_j, \, |\lambda_j| \approx L_j \}}\big\|_{[3, \RR \times \RR]} \\
      &\lesa N_3^\epsilon \Big(  N_3^{(\frac{1}{2} - b_{max})_+ - b_{med} - b_{min}}N^{\frac{1}{2}}_{min}  + N_3^{-b_{min}} N_{min}^{(\frac{1}{2} - b_{max})_+ + (\frac{1}{2} - b_{med})_+}\Big) \\
      &\lesa N_3^\epsilon \Big(  N_3^{\frac{1}{2} - b_1 - b_2 - b_3}N^{\frac{1}{2}}_{min}  + N_3^{-b_{min}} N_{min}^{(\frac{1}{2} - b_{max})_+ + (\frac{1}{2} - b_{med})_+}\Big)
      \end{align*}
where we used the inequality
            \begin{equation}\label{lem dyadic multiplier estimate - eqn 1} N^{\frac{1}{2}}_{min} N^{(\frac{1}{2} - b_{max})_+ - b_{med} - b_{min}}_3 \les N^{\frac{1}{2}}_{min} N_3^{\frac{1}{2} - b_1 - b_2 - b_3} + N^{(\frac{1}{2} - b_{max})_+ + (\frac{1}{2} - b_{med})_+}_{min} N_3^{-b_{min}} .\end{equation}
which is trivial if $b_{max}<\frac{1}{2}$. On the other hand, if $b_{max} \g \frac{1}{2}$, then (\ref{lem dyadic multiplier estimate - eqn 1}) follows by noting that since $b_j + b_k > 0$ we have $b_{med} > 0$ and so
         $$ N_{min}^{\frac{1}{2} } N_{3}^{-b_{med} - b_{min}} \les N^{\frac{1}{2} - b_{med}}_{min} N_3^{-b_{min}} \les
          N_{min}^{(\frac{1}{2} - b_{med})_+} N_3^{-b_{min}}$$
as required. \\

\textbf{ $\bullet$ Case 1b ($L_{med} \les N_3$ and $L_{max} = L_3$).}  Lemma \ref{lem dyadic multiplier estimate} together with the assumption $L_{max} = L_3$ gives
        $$ \big\|  \Pi_{j=1}^3 \ind_{\{|\xi_j|\approx N_j, \, |\lambda_j| \approx L_j \}}\big\|_{[3, \RR \times \RR]}
                        \lesa L_{min}^{\frac{1}{2}} N_{min}^{\frac{1}{2}}.$$
 Suppose $L_1 \les L_2$. Then
    \begin{align}
       \sum_{L_{max} \approx N_3 \gtrsim L_{med} } L^{-b_1}_1 L^{-b_2}_2 L^{-b_3}_3 \big\|  \Pi_{j=1}^3 &\ind_{\{|\xi_j|\approx N_j, \, |\lambda_j| \approx L_j \}}\big\|_{[3, \RR \times \RR]}\label{lem dyadic multiplier estimate - eqn 2} \\
        &\lesa N_{min}^{\frac{1}{2}} N_3^{-b_3} \sum_{L_2 \les N_3 } L_2^{ - b_2} \sum_{L_1\les L_2} L_1^{\frac{1}{2} - b_1} \notag \\ \notag
        &\lesa N_{min}^{\frac{1}{2}} N_3^{-b_3} \sum_{L_2 \les N_3 } L_2^{ (\frac{1}{2}-b_1)_+ - b_2} \log(L_2) \\ \notag
        &\lesa  N_{min}^{\frac{1}{2}} N_3^{((\frac{1}{2} - b_1)_+ - b_2)_+ - b_3 + \epsilon}
    \end{align}
 for any $\epsilon>0$. If we have
        \begin{equation}
          \label{lem dyadic multiplier estimate - eqn 3}
            N_{min}^{\frac{1}{2}} N_3^{((\frac{1}{2} - b_1)_+ - b_2)_+ - b_3}\les
            N_3^{\frac{1}{2} - b_1 - b_2 - b_3} N_{min}^\frac{1}{2} + N_{3}^{-b_3} N_{min}^{\frac{1}{2}} + N_3^{-b_{min}} N_{min}^{(\frac{1}{2} - b_{max})_+ + (\frac{1}{2} - b_{med})_+}
        \end{equation}
then we get
    $$ (\ref{lem dyadic multiplier estimate - eqn 2}) \lesa N_3^\epsilon \Big(  N_3^{\frac{1}{2} - b_1 - b_2 - b_3} N_{min}^\frac{1}{2} + N_{3}^{-b_3} N_{min}^{\frac{1}{2}} + N_3^{-b_{min}} N_{min}^{(\frac{1}{2} - b_{max})_+ + (\frac{1}{2} - b_{med})_+}\Big)$$
as required. The case $L_1 \g L_2$ follows an identical argument and so it remains to show (\ref{lem dyadic multiplier estimate - eqn 3}). To this end note that if $(\frac{1}{2} - b_1)_+ - b_2 <0$ then we simply have
        $$ N_{min}^{\frac{1}{2}} N_3^{((\frac{1}{2} - b_1)_+ - b_2)_+ - b_3} = N_{min}^{\frac{1}{2}} N_3^{- b_3}. $$
On the other hand, if $(\frac{1}{2} - b_1 )_+ - b_2 \g 0$, then by using (\ref{lem dyadic multiplier estimate - eqn 1}) we have
    \begin{align*}
        N_{min}^{\frac{1}{2}} N_3^{((\frac{1}{2} - b_1)_+ - b_2)_+ - b_3}&=N^{\frac{1}{2}}_{min} N_3^{(\frac{1}{2} - b_1)_+ - b_2 - b_3}\\
                        &\les N^{\frac{1}{2}}_{min} N_3^{(\frac{1}{2} - b_{max})_+ - b_{med} - b_{min}}\\
                        &\les  N^{\frac{1}{2}}_{min} N_3^{\frac{1}{2} - b_1 - b_2 - b_3} + N^{(\frac{1}{2} - b_{max})_+ + (\frac{1}{2} - b_{med})_+}_{min} N_3^{-b_{min}}
    \end{align*}
and so we obtain (\ref{lem dyadic multiplier estimate - eqn 3}).\\

\textbf{$\bullet$ Case 2 ($ L_{med} \g N_3$).} In this case we have $L_{max} \approx L_{med}$ and by Lemma \ref{lem dyadic multiplier estimate}
        $$ \big\|  \Pi_{j=1}^3 \ind_{\{|\xi_j|\approx N_j, \, |\lambda_j| \approx L_j \}}\|_{[3, \RR \times \RR]} \lesa N^{\frac{1}{2}}_{min} L^{\frac{1}{2}}_{min} .$$
Suppose $L_1 \g L_2 \g L_3$. Then
    \begin{align*}
      \sum_{L_{max} \approx L_{med} \gtrsim N_3} L^{-b_1}_1 L^{-b_2}_2 L^{-b_3}_3 N^{\frac{1}{2}}_{min} L^{\frac{1}{2}}_{min}
        &\approx  N_{min}^{\frac{1}{2}} \sum_{L_2 \gtrsim N_3} L^{-b_1 -b_2}_2 \sum_{L_3 \les L_2} L^{\frac{1}{2} - b_3}_3 \\
        &\lesa N^{\frac{1}{2}}_{min} \sum_{L_2 \gtrsim N_3} L^{(\frac{1}{2} - b_3)_+ \, -b_1 -b_2}_2 \log(L_2) \\
        &\lesa N^{\frac{1}{2}}_{min} N^{ (\frac{1}{2} - b_3)_+ - b_1 - b_2 + \epsilon}_3\\
        &\lesa N^{\frac{1}{2}}_{min} N^{ (\frac{1}{2} - b_{max})_+ - b_{med} - b_{min} + \epsilon}_3
    \end{align*}
provided $b_1  + b_2 + b_3 > \frac{1}{2}$, $b_j + b_k > 0$, and we choose $\epsilon>0$  sufficiently small.
Since this argument also holds for all other size combinations of the $L_j$, we get from (\ref{lem dyadic multiplier estimate - eqn 1})
        \begin{align*}\sum_{L_{max} \approx L_{med} \gtrsim N_3} L^{-b_1}_1 L^{-b_2}_2 L^{-b_3}_3 \big\|  \Pi_{j=1}^3 &\ind_{\{|\xi_j|\approx N_j, \, |\lambda_j| \approx L_j \}}\big\|_{[3, \RR \times \RR]}\\
            &\lesa N_3^\epsilon\Big( N_3^{\frac{1}{2} - b_1 - b_2 - b_3} N^{\frac{1}{2}}_{min}  + N_3^{-b_{min}} N^{(\frac{1}{2} - b_{max})_+ + (\frac{1}{2} - b_{med})_+}_{min}\Big) \end{align*}
and so lemma follows.

\end{proof}
\end{lemma}

We now come to the proof of Theorem \ref{thm main X^sb esimate}.

\begin{proof}[Proof of Theorem \ref{thm main X^sb esimate}] By Lemma \ref{lem reduction to dyadic pieces} and Lemma \ref{lem inner summation} it suffices to estimate the sum
        $$ \sup_{N} \sum_{N_{max} \approx N_{med} \approx N}  \Big(\Pi_{j=1}^3 N_{j}^{-s_j}\Big)  N_{min}^{\alpha} N_3^{-\beta} $$
for the pairs
    $$(\alpha, \beta) \in \Bigg\{  \Bigg(\frac{1}{2}, \,\,\,b_1 + b_2 + b_3 - \frac{1}{2} - \epsilon\Bigg),\,\,\, \Bigg( \frac{1}{2}, \,\,\,b_3 - \epsilon \Bigg), \,\,\,\Bigg( \Big(\frac{1}{2} - b_{max}\Big)_+ + \Big( \frac{1}{2} - b_{med}\Big)_+,\,\,\, b_{min}  - \epsilon \Bigg) \Bigg\} $$
where $\epsilon>0$ may be taken arbitrarily small. Let $s_1' = s_1$, $s_2' = s_2$, and $s_3' = s_3 + \beta$. Then we have to show
            $$ \sup_{N} \sum_{N_{max} \approx N_{med} \approx N}  \Big(\Pi_{j=1}^3 N_{j}^{-s_j'}\Big)  N_{min}^{\alpha}  < \infty.$$
    Since this summation is symmetric with respect to the $N_j$, we may assume $ N_1 \les N_2 \les N_3$. Then
            \begin{align*}
               \sum_{N_{max} \approx N_{med} \approx N}  \Big(\Pi_{j=1}^3 N_{j}^{-s_j'}\Big)  N_{min}^{\alpha}
                                &\lesa N^{-s_2' - s_3'} \sum_{N_1 \les N} N_1^{ -s_1'  + \alpha} <\infty
            \end{align*}
    provided $s_j'  + s_k' \g 0$ and $s_1' + s_2' + s_3' > \alpha$. These conditions hold by the assumptions in Theorem \ref{thm main X^sb esimate} provided we choose $\epsilon$ sufficiently small.

\end{proof}

\section{Counter Examples}\label{sec counter examples}

Here we prove that the conditions in Theorem \ref{thm main X^sb esimate} are sharp up to equality.

\begin{proposition}\label{prop - counter examples}
Assume the estimate (\ref{thm main X^sb est - main trilinear est}) holds. Then we must have
\begin{equation} b_j + b_k \geqslant 0, \qquad b_1+ b_2 + b_3  \geqslant \frac{1}{2} \label{prop - counter examples - b cond}\end{equation}
    and for $k \in \{ 1, 2\}$
      \begin{align} \label{prop - counter examples - s cond 1} s_1 + s_2 &\geqslant 0, \\
       s_k + s_3  &\geqslant -b_{min},\label{prop - counter examples - s cond 2}\\
       s_k +s_3&\geqslant \frac{1}{2} - b_1 - b_2 - b_3, \label{prop - counter examples - s cond 3}\\
       s_1 + s_2 + s_3 &\geqslant  \frac{1}{2} - b_3,\label{prop - counter examples - s cond 4}\\
       s_1 + s_2 + s_3 &\geqslant \Big(\frac{1}{2} - b_{max}\Big)_+ + \Big(\frac{1}{2} - b_{med}\Big)_+ - b_{min}. \label{prop - counter examples - s cond 5}\end{align}
\end{proposition}

\begin{remark}
   We note that in some regions the $\pm$ structure in (\ref{thm main X^sb esimate}) is redundant and so the counter examples for the
   Wave-Sobolev spaces used in \cite{D'Ancona2010} and \cite{Selberg2008} would apply. In fact, the counterexamples in \cite{D'Ancona2010} already essentially show that we must have (\ref{prop - counter examples - b cond}), (\ref{prop - counter examples - s cond 1}), and  (\ref{prop - counter examples - s cond 5}). On the other hand, the conditions (\ref{prop - counter examples - s cond 2} - \ref{prop - counter examples - s cond 4}) reflect the $\pm$ structure and thus cannot be deduced from \cite{D'Ancona2010}.
\end{remark}

\begin{proof} It suffices to find necessary conditions for the estimate (\ref{main est fourier side}).  Moreover we may assume $\pm = +$ since the case $\pm = - $ follows by a reflection in the $\tau_j$ variables.  Let $\lambda \gg 1$ be some large parameter. The main idea is as follows. Assume we have sets $A, B, C \subset \RR^{1+1}$ with
        \begin{equation}
          \label{prop - counter examples - set cond 1}
            |A|\approx \lambda^{d_1}, \qquad |B| \approx \lambda^{d_2}, \qquad |C|\approx \lambda^{d_3}.
        \end{equation}
Moreover, suppose that if $(\tau_2, \xi_2 ) \in B$ and $ (\tau_3, \xi_3) \in C$, then
        \begin{equation}\label{prop - counter examples - set cond 2}
              - ( \tau_2 + \tau_3, \xi_2 + \xi_3) \in A
        \end{equation}
and
        \begin{equation}\label{prop - counter examples - set cond 3} \frac{ \al \xi_2 + \xi_3 \ar^{-s_1} \al \xi_2 \ar^{-s_2} \al \xi_3 \ar^{-s_3} }{ \al \tau_2 +\tau_3 +  \xi_2 + \xi_3 \ar^{b_1} \al \tau_2 + \xi_2 \ar^{b_2}\al \tau_3 - \xi_3 \ar^{b_3}} \approx \lambda^{ - \delta}.
        \end{equation}
Let $f_1 = \ind_A$, $f_2 = \ind_B$, $f_3 = \ind_C$. Then using the conditions (\ref{prop - counter examples - set cond 1} - \ref{prop - counter examples - set cond 3}) we have
        \begin{align*}
            \int_{\Gamma} \mathfrak{m}(\tau, \xi) \Pi_{j=1}^3 f_j(\tau_j, \xi_j) d\sigma(\tau, \xi) &\gtrsim \lambda^{-\delta} \int_B \int_C d\tau_3d\xi_3d\tau_2d\xi_2 \\
            &\approx \lambda^{d_2 + d_3 - \delta}.
        \end{align*}
Therefore, assuming that the inequality (\ref{main est fourier side}) holds, we must have
        $$ \lambda^{d_2 + d_3 - \delta} \lesa |A|^{\frac{1}{2}} |B|^{\frac{1}{2}} |C|^{\frac{1}{2}} \approx \lambda^{\frac{d_1 + d_2 + d_3}{2}}.$$
By choosing $\lambda$ large, we then derive the necessary condition
            \begin{equation}\label{prop - counter examples - necessary cond} \delta + \frac{ d_1 - d_2 - d_3}{2} \g 0.\end{equation}
Thus it will suffice to find sets $A$, $B$, and $C$ satisfying the conditions (\ref{prop - counter examples - set cond 1} - \ref{prop - counter examples - set cond 3}) with particular values of $\delta$, $d_1$, $d_2$, and $d_3$. \\

\textbf{$\bullet$ Necessity of (\ref{prop - counter examples - b cond}).}  
We first show that $b_j + b_k \g 0$. Since the estimate (\ref{main est fourier side}) is symmetric in $b_1$, $b_2$, it suffices to consider the pairs $(j, k) \in \{(1, 2), \, (1, 3)\}$. For the first pair, we choose
    $$ B= \{|\tau +\lambda|\les 1, \,\,\,|\xi|\les 1 \}, \qquad C = \{|\tau | \les 1, \,\,\, |\xi|\les 1 \}, \qquad A = \{|\tau-\lambda|\les 2, \,\,\, |\xi|\les 2\} .$$
Then the conditions (\ref{prop - counter examples - set cond 1} - \ref{prop - counter examples - set cond 3}) hold with $d_1 = d_2 = d_3 = 0$ and $\delta = b_1 + b_2$ and so from (\ref{prop - counter examples - necessary cond}) we obtain the necessary condition
            $b_1 + b_2 \g 0.$

On the other hand, for the pair $(1, 3)$ we choose
        $$ B= \{|\tau |\les 1,\,\,\,|\xi|\les 1\}, \qquad C = \{|\tau + \lambda| \les 1,  \,\,\,\,|\xi|\les 1  \}, \qquad A = \{|\tau-\lambda|\les 2, \,\,\, |\xi|\les 2 \} .$$
Then as in the previous case, the conditions (\ref{prop - counter examples - set cond 1} - \ref{prop - counter examples - set cond 3}) hold with $d_1 = d_2 = d_3 = 0$ and $\delta = b_1 + b_3$ and so from (\ref{prop - counter examples - necessary cond}) we obtain the necessary condition $b_1 + b_3 \g 0.$

To show the second condition in (\ref{prop - counter examples - b cond}) is also necessary, we take
    $$ B= \{ |\tau - 2\lambda|\les \lambda, \,\,\,  |\xi|\les 1\}, \qquad C = \{ |\tau - 2\lambda| \les \lambda, \,\,\,\,   |\xi|\les 1\}, \qquad A = \{ |\tau+ 4\lambda|\les 2\lambda, \,\,\,|\xi|\les 2 \} .$$
Then  (\ref{prop - counter examples - set cond 1} - \ref{prop - counter examples - set cond 3}) hold with $d_1=d_2=d_3 =1$ and $\delta = b_1 +b_2 + b_3$ which leads to the condition $ b_1 + b_2  +b_3 \g \frac{1}{2}$.\\

\textbf{$\bullet$ Necessity of (\ref{prop - counter examples - s cond 1}).} Let
    $$ B= \{ |\tau - \lambda|\les 1, \,\,\,  |\xi + \lambda|\les 1\}, \qquad C = \{ |\tau | \les 1, \,\,\,\,   |\xi|\les 1\}, \qquad A = \{ |\tau + \lambda|\les 2, \,\,\,|\xi - \lambda|\les 2 \} .$$
Then  (\ref{prop - counter examples - set cond 1} - \ref{prop - counter examples - set cond 3}) hold with $d_1=d_2=d_3 =0$ and $\delta =s_1 + s_2$ and so we must have (\ref{prop - counter examples - s cond 1}).\\

\textbf{$\bullet$ Necessity of (\ref{prop - counter examples - s cond 2}).} By symmetry we may assume $k=1$. Suppose $b_{min} = b_1$ and choose
    $$ B= \{ |\tau |\les 1, \,\,\,  |\xi|\les 1\}, \qquad C = \{ |\tau - \lambda | \les 1, \,\,\,\,   |\xi - \lambda|\les 1\}, \qquad A = \{ |\tau + \lambda|\les 2, \,\,\,|\xi+\lambda|\les 2 \} .$$
Then  (\ref{prop - counter examples - set cond 1} - \ref{prop - counter examples - set cond 3}) hold with $d_1=d_2=d_3 =0$ and $\delta =s_1 + s_3 + b_1$ and so we must have $ s_1 + s_3 + b_1 \g 0.$

On the other hand, if $b_{min} = b_2$ we let
        $$ B= \{ |\tau + 2\lambda|\les 1, \,\,\,  |\xi |\les 1\}, \qquad C = \{ |\tau -  \lambda | \les 1, \,\,\,\,   |\xi - \lambda|\les 1\}, \qquad A = \{ |\tau -\lambda |\les 2, \,\,\,|\xi + \lambda|\les 2 \} .$$
Then  (\ref{prop - counter examples - set cond 1} - \ref{prop - counter examples - set cond 3}) hold with $d_1=d_2=d_3 =0$ and $\delta =s_1 + s_3 + b_2$ and so we obtain the condition $ s_1 + s_3 + b_2 \g 0.$

The final case, $b_{min} = b_3$, follows by taking
    $$ B= \{ |\tau |\les 1, \,\,\,  |\xi |\les 1\}, \qquad C= \{ |\tau  - \lambda | \les 1, \,\,\,\,   |\xi + \lambda|\les 1\}, \qquad A = \{ |\tau + \lambda|\les 2, \,\,\,|\xi-\lambda|\les 2 \} .$$
Again the conditions (\ref{prop - counter examples - set cond 1} - \ref{prop - counter examples - set cond 3}) hold with $d_1=d_2=d_3 =0$ and $\delta =s_1 + s_3 + b_3$. Hence (\ref{prop - counter examples - s cond 2}) is necessary. \\

\textbf{$\bullet$ Necessity of (\ref{prop - counter examples - s cond 3}).}
As in the previous case, by symmetry, we may assume $k=1$. Let
     $$ B= \Big\{ |\tau - \lambda | \les \frac{\lambda}{4} , \,\,\,\,   |\xi|\les 1\Big\}, \qquad C =\Big\{ |\tau |\les \frac{\lambda}{4} , \,\,\,|\xi -\lambda|\les \frac{\lambda}{4} \Big\}, \qquad A = \Big\{ |\tau + \lambda|\les \frac{\lambda}{2}, \,\,\,  |\xi + \lambda|\les \frac{\lambda}{2} \Big\} .$$
Then  (\ref{prop - counter examples - set cond 1} - \ref{prop - counter examples - set cond 3}) hold with $d_1=d_3=2$, $d_2=1$,  and $\delta =s_1 + s_3 + b_1 + b_2 + b_3$. Thus we obtain the necessary condition (\ref{prop - counter examples - s cond 3}).\\

\textbf{$\bullet$ Necessity of (\ref{prop - counter examples - s cond 4}).}
In this case we choose
    $$ B= \Big\{ |\tau  + \xi|\les 1, \,\,\,  |\xi - \lambda|\les \frac{\lambda}{4} \Big\},  \qquad C = \Big\{ |\tau + \xi |\les 1, \,\,\,|\xi -\lambda|\les \frac{\lambda}{4} \Big\}, \qquad A=\Big\{ |\tau + \xi|\les 2, \,\,\,  |\xi + 2 \lambda|\les \frac{\lambda}{2} \Big\}.$$
Then a simple computation shows that (\ref{prop - counter examples - set cond 1} - \ref{prop - counter examples - set cond 3}) hold with $d_1=d_2=d_3=1$, and $\delta =s_1 + s_2 + s_3 + b_3$. So we see that (\ref{prop - counter examples - s cond 4}) is necessary.\\

\textbf{$\bullet$ Necessity of (\ref{prop - counter examples - s cond 5}).}
We break this into the 3 conditions
    \begin{equation}\label{prop - counter examples - split into 3 cond} s_1 + s_2 + s_3 \g 1 - b_1 - b_2 - b_3, \qquad s_1 + s_2 + s_3\g \frac{1}{2} - b_j - b_k, \qquad s_1 + s_2 + s_3 \g - b_{min}. \end{equation}
For the first inequality, we take
     $$ B= \Big\{ |\tau|\les \frac{\lambda}{4}, \,\,\,  |\xi - \lambda|\les \frac{\lambda}{4} \Big\},  \qquad C = \Big\{ |\tau  |\les \frac{\lambda}{4}, \,\,\,|\xi  - \lambda|\les \frac{\lambda}{4} \Big\}, \qquad A=\Big\{ |\tau|\les \frac{\lambda}{2}, \,\,\,  |\xi + 2\lambda|\les \frac{\lambda}{2} \Big\}.$$
Then we have (\ref{prop - counter examples - set cond 1} - \ref{prop - counter examples - set cond 3}) with $d_1=d_2=d_3=2$, and $\delta =s_1 + s_2 + s_3 +b_1 + b_2+ b_3$. Therefore we must have $ s_1+s_2+s_3  \g 1 - b_1 - b_2 - b_3.$

We now consider the second inequality in (\ref{prop - counter examples - split into 3 cond}). By symmetry, it suffices to consider $(j, k) \in \{ (1, 2), (1, 3)\}$. Let
    $$ B= \Big\{ |\tau + \xi - \lambda|\les \frac{\lambda}{4}, \,\,\,  |\xi - \lambda|\les \frac{\lambda}{4} \Big\}, \,\,\,\, C = \Big\{ |\tau - \xi  |\les 1, \,\,\,|\xi  - \lambda|\les \frac{\lambda}{4} \Big\}, \,\,\,\, A=\Big\{ |\tau + \xi + 3\lambda|\les \lambda, \,\,\,  |\xi + 2\lambda|\les \frac{\lambda}{2} \Big\}.$$
Then  (\ref{prop - counter examples - set cond 1} - \ref{prop - counter examples - set cond 3}) hold with $d_1=d_2=2$, $d_3=1$,  and $\delta =s_1 + s_2 + s_3 +b_1 + b_2$. Therefore we must have $ s_1+s_2+s_3 >\frac{1}{2} - b_1 - b_2.$ On the other hand, for the case $(j, k) = (1, 3)$, we take
    $$ B= \Big\{ |\tau + \xi |\les 1, \,\,\,  |\xi - \lambda|\les \frac{\lambda}{4} \Big\},  \qquad C = \Big\{ |\tau   |\les \frac{\lambda}{4}, \,\,\,|\xi  - \lambda|\les \frac{\lambda}{4} \Big\}, \qquad A=\Big\{ |\tau + \xi + \lambda|\les \frac{3\lambda}{4}, \,\,\,  |\xi + 2\lambda|\les \frac{\lambda}{2} \Big\}.$$
A simple computation shows that (\ref{prop - counter examples - set cond 1} - \ref{prop - counter examples - set cond 3}) are satisfied with $d_1=d_3=2$, $d_2=1$,  and $\delta =s_1 + s_2 + s_3 +b_1 + b_3$.

Finally, the third condition in (\ref{prop - counter examples - split into 3 cond}) follows from the conditions  (\ref{prop - counter examples - s cond 1}) and (\ref{prop - counter examples - s cond 2}).

\end{proof}

\providecommand{\bysame}{\leavevmode\hbox to3em{\hrulefill}\thinspace}
\providecommand{\MR}{\relax\ifhmode\unskip\space\fi MR }
\providecommand{\MRhref}[2]{%
  \href{http://www.ams.org/mathscinet-getitem?mr=#1}{#2}
}
\providecommand{\href}[2]{#2}

\end{document}